\documentclass[12pt]{amsart}
\usepackage[latin1]{inputenc}
\usepackage{amssymb, amsmath, amscd, mathrsfs,marvosym, psfrag} 

\usepackage{pstricks}

\input xy
\xyoption{all}
\DeclareMathAlphabet{\mathpzc}{OT1}{pzc}{m}{it}

\newtheorem{theorem}{Theorem}[section]

\newtheorem{proposition}[theorem]{Proposition}
\newtheorem{corollary}[theorem]{Corollary}

\newtheorem{conjecture}[theorem]{Conjecture}

\newtheorem{lemma}[theorem]{Lemma}

\newtheorem*{theoremnonum}{Theorem}

\theoremstyle{definition}
\newtheorem{definition}[theorem]{Definition}

\theoremstyle{remark}
\newtheorem{remark}[theorem]{Remark}

\def\varle{\leqslant}

\newcommand{\CB}{{\mathcal B}}

\newcommand{\CE}{{\mathcal E}}

\newcommand{\CG}{{\mathcal G}}

\newcommand{\CI}{{\mathcal I}}
\newcommand{\CJ}{{\mathcal J}}

\newcommand{\CM}{{\mathcal M}}
\newcommand{\CN}{{\mathcal N}}

\newcommand{\CS}{{\mathcal S}}

\newcommand{\CV}{{\mathcal V}}
\newcommand{\CW}{{\mathcal W}}

\newcommand{\CX}{{\mathcal X}}

\newcommand{\CY}{{\mathcal Y}}
\newcommand{\CZ}{{\mathcal Z}}

\newcommand{\SB}{{\mathscr B}}

\newcommand{\SM}{{\mathscr M}}

\newcommand{\SZ}{{\mathscr Z}}

\newcommand{\hCW}{{\widehat\CW}}
\newcommand{\hCG}{{\widehat\CG}}

\newcommand{\hCS}{{\widehat\CS}}

\newcommand{\hV}{{\widehat V}}
\newcommand{\hR}{{\widehat R}}
\newcommand{\hX}{{\widehat X}}

\newcommand{\hH}{{\widehat H}}

\newcommand{\hbH}{{\widehat{\bH}}}

\newcommand{\hw}{\widehat w}

\newcommand{\DZ}{{\mathbb Z}}

\newcommand{\DN}{{\mathbb N}}
\newcommand{\DQ}{{\mathbb Q}}

\newcommand{\DF}{{\mathbb F}}

\newcommand{\bH}{{\mathbf H}}

\newcommand{\bs}{{\mathbf s}}

\newcommand{\height}{{\operatorname{ht}}}
\newcommand{\ch}{{\operatorname{char}\, }}

\newcommand{\Hom}{{\operatorname{Hom}}}

\newcommand{\supp}{{\operatorname{supp}}}

\newcommand{\coker}{{\operatorname{coker}}}
\newcommand{\im}{{\operatorname{im}}}

\newcommand{\dual}{{\mathsf D}}

\newcommand{\rk}{{{\operatorname{rk}}}}

\newcommand{\ol}{\overline}
\newcommand{\ul}{\underline}

\newcommand{\ev}{\operatorname{ev}}
\newcommand{\id}{{\operatorname{id}}}

\newcommand{\Quot}{{\operatorname{Quot\,}}}

\newcommand{\linie}{{\,\text{---\!\!\!---}\,}}
\newcommand{\llinie}{{\text{---\!\!\!---\!\!\!---}}}

\newcommand{\comment}[1]{}

\begin{document}

\pagenumbering{arabic}
\title[An upper bound on exceptional characteristics]{An upper bound on the exceptional characteristics for Lusztig's character
formula}
\author[]{Peter Fiebig}
\thanks{The author is partially supported by a grant of the
  Landesstiftung Baden--W\"urttemberg }

\begin{abstract} We develop and study a Lefschetz theory in a
  combinatorial category associated to a root system and derive an 
  upper bound on the exceptional characteristics for Lusztig's
  formula  for the simple rational characters of a
  reductive algebraic group. Our bound is huge compared to the Coxeter
  number. It is, however, given by an explicit formula.
\end{abstract}

\address{Mathematisches Institut, Universit{ä}t Freiburg, 79104 Freiburg, Germany}
\email{peter.fiebig@math.uni-freiburg.de}
\maketitle

\section{Introduction}
Lusztig's conjecture on the irreducible rational characters of a reductive
algebraic group over a field $k$ of  positive characteristic has in the last 20 years
been approached from various directions. It is known that Lusztig's
formula holds for large enough characteristics (with respect to a
fixed root system), yet so far it is unknown 
what ``large enough'' means in all but low rank cases (conjecturally, the characteristic should
be at least the Coxeter number of the root system). So in
almost all explicitely given cases the irreducible  characters of a reductive
group  are unknown.  This is  certainly not a completely satisfying
situation. 

\subsection{Lusztig's conjecture as a moment graph problem} The most recent approach towards Lusztig's conjecture, contained in
the articles \cite{FieModRep} and \cite{FieCharIrr}, gives a
connection between representations of reductive groups over $k$ (more precisely, representations of their Lie
algebras) and the theory of $k$-sheaves on moment graphs. In particular,
it is shown in the above papers that Lusztig's conjecture follows from
a similar multiplicity conjecture for the Braden-MacPherson sheaves
(with coefficients in the field $k$) on
an  affine moment graph.

The moment graph theory has the advantage that it can be formulated
and studied in a relatively elementary way. In particular, one can
determine the $p$-smooth locus on an affine moment graph by quite
elementary arguments (cf.~\cite{FieZdual}). The result is in
complete accordance with Lusztig's conjecture and  yields its
multiplicity one case for all fields with characteristic at least the
Coxeter number.  Unfortunately, this cannot be reinterpreted in terms of characters. 

\subsection{A Lefschetz theory for sheaves on moment graphs}
In the present paper we study a Lefschetz theory for  moment graphs and  show
that the multiplicity conjecture is implied by a Hard Lefschetz
conjecture for the Braden-MacPherson sheaves. Then we prove the latter
conjecture in the characteristic zero case following the argumentation
in  \cite{SoeGAFA}, where the equivalent category of {\em special
  bimodules} is considered. The essential ingredient for the proof is
the fact that our objects occur as  equivariant
intersection cohomologies of Schubert varieties.  The
Hard Lefschetz conjecture for moment graphs  then follows from the Hard
Lefschetz theorem for projective varieties.

\subsection{From $\DQ$ to $\DF_p$ via $\DZ$}

In positive characteristic, Soergel's approach does not work any
more for several reasons, one of them being that the correspondence between Braden-MacPherson sheaves  and intersection
cohomology cannot be established so far. Yet one can deduce the
conjecture for big enough characteristics from its characteristic zero
analogue. For this one first replaces
the Braden-MacPherson sheaves by their global sections which naturally
are modules over the (commutative, associative) {\em structure algebra}
$\CZ$ associated to the graph. Then one translates the multiplicity conjecture into the
language of $\CZ$-modules. The advantage is that there is an
alternative way to construct those global sections as direct summands
of Bott--Samelson modules. These turn out to be defined over
the integers, so that we  can apply a base change argument. A closer
look at the base change procedure  opens a way to give an
estimate on the exceptional primes. 

In the particular case that is
connected to Lusztig's conjecture we get the following result. For a
fixed root system $R$ we let $\hw_0$ be  the element in the affine Weyl group corresponding to
the lowest alcove in the anti-fundamental box and we let  $l=l(\hw_0)$
be its length. For any reduced expression $\bs$ of $\hw_0$ we define the numbers $r=r(\bs)$, $d=d(\bs)$ and
$N=N(\bs)$ explicitly in terms of the affine Hecke algebra (cf.~Section
\ref{subsec-maintheo}). Then we set
$$
U(\hw_0):=\min_{\bs}r!(r!(r-1)!N^{l(\hw_0)+2d})^r.
$$

\begin{theoremnonum} Suppose that $\ch
  k>U(\hw_0)$. Then the Hard Lefschetz
  conjecture for the Braden-MacPherson sheaves, hence Lusztig's
  conjecture for the algebraic groups with root system $R$, holds. 
\end{theoremnonum}
 Admittedly, the above bound is very far from the Coxeter number. It
is, however, an explicit number and in the general case the only bound
that is known  to the author.  

After having established the Lefschetz theory, we derive the above estimate using  the most basic linear algebra.  I am sure that any careful reader can  immediately come up with an improved bound. I strongly doubt,
however, that the methods outlined in this paper suffice to come
anywhere in the vicinity of the  Coxeter number.

\subsection{Acknowledgements} Parts of the research for this paper were done
 at the Osaka City University and at the Emmy Noether center in Erlangen. I am grateful towards both institutions for their generous support.

\section{Sheaves on the affine Bruhat graph}

In this section we introduce the {\em
  Braden--MacPherson sheaves} on  affine Bruhat graphs. We state the
main conjecture on the graded rank of their stalks and quickly review
the connection to Lusztig's conjecture on simple rational characters
of reductive groups. As a first step, let us recall the notion of a
moment graph.

\subsection{Moment graphs}

Let $Y\cong\DZ^r$ be a lattice. 
\begin{definition} A {\em moment graph} $\CG=(\CV,\CE,\alpha,\varle)$ over $Y$ is given by the
  following data.
\begin{enumerate} 
\item A graph $(\CV,\CE)$ with set of vertices $\CV$ and set of edges
  $\CE$.
\item A map $\alpha\colon \CE\to Y\setminus\{0\}$, called the {\em
    labelling}.
\item A partial order ``$\varle$'' on $\CV$ such that two vertices are
  comparable if they are connected by an edge.
\end{enumerate}
\end{definition}

We use the notation $E\colon x\linie y$ for an edge $E$ connecting $x,y\in\CV$.
The partial order allows us to endow each edge with a direction. We
write $E\colon x\to y$ if $x<y$. We write $E\colon
x\stackrel{\alpha}\llinie y$ or $E\colon x\stackrel{\alpha}\to y$ if
we want to specify the label $\alpha=\alpha(E)$ of $E$.

Each subset $\CI\subset \CV$ defines a (full) sub-moment graph $\CG_I$ of
$\CG$ in the obvious way. For $w\in \CV$ we denote
by $\CG_{\le w}$ the  sub-moment graph associated to the set
$\{x\in\CV\mid x\le w\}$. In the following we write $\{\le w\}$ instead of
$\{x\in\CV\mid x\le w\}$, and we write $\{<w\}$, $\{>w\}$, etc.~ for
the similarly defined sets.

We will mostly deal in this article with the moment graphs that are associated to an
affine root system. Let us quickly introduce the notions and notations
that we need.

\subsection{Finite root systems} 
Let $V$ be a finite dimensional $\DQ$-vector space and let $R\subset
V$ be an irreducible, reduced, 
finite root system in the sense of Bourbaki (cf.~\cite{Bourbaki}). We denote by $V^\ast=\Hom(V,\DQ)$ the dual space, and write $\langle\cdot,\cdot\rangle\colon V\times V^\ast\to \DQ$ for the natural pairing. For $\alpha\in R$ we denote by $\alpha^\vee\in V^\ast$
its coroot and we let $R^\vee:=\{\alpha^\vee\mid\alpha\in R\}\subset
V^\ast$ be the dual root system. Let
$$
X:=\{\lambda\in V\mid
\langle\lambda,\alpha^\vee\rangle\in\DZ\text{ for all $\alpha\in
  R$}\}
$$
be the weight lattice and
$$
X^\vee:=\{v\in V^\ast\mid \langle\alpha,v\rangle\in\DZ\text{ for all
  $\alpha\in R$}\}
$$
the coweight lattice. We have $\DZ R\subset X$ and $\DZ R^\vee\subset
X^\vee$, and $(\DZ R, X^\vee)$ and $(X,\DZ R^\vee)$ are dual lattices
under the  pairing $\langle\cdot,\cdot\rangle$.

\subsection{The affine Weyl group}
For $\alpha\in R$ and $n\in\DZ$ we define the affine reflection
$s_{\alpha,n}\colon V^\ast\to V^\ast$ by 
$$
s_{\alpha,n}(v):=v-(\langle\alpha,v\rangle-n)\alpha^\vee.
$$
The affine Weyl group $\hCW$ is the subgroup of affine transformations
on $V^\ast$ that is generated by $\{s_{\alpha,n}\}_{\alpha\in
  R,n\in\DZ}$. Since $(-\alpha)^\vee=-\alpha^\vee$ we have $s_{\alpha,n}=s_{-\alpha,-n}$.

Adding an additional dimension allows us to linearize the above affine
action. If we set $\hV:=V\oplus\DQ$, then we have $\hV^\ast=V^\ast\oplus\DQ$ where we extend the natural pairing between $V$ and $V^\ast$ in the obvious way: $\langle (\lambda,m),
(v,n)\rangle=\lambda(v)+mn$. Let
$s_{\alpha,n}$ act on $\hV^\ast$ by the formula 
$$
s_{\alpha,n}(v,m):=(v-(\langle\alpha,v\rangle-mn)\alpha^\vee,m).
$$
This extends to a linear action of $\hCW$ on $\hV^\ast$. The level spaces
$\hV^\ast_m:=\{(v,m)\mid v\in V^\ast\}$ are stabilized by $\hCW$, and
on the affine space $\hV^\ast_1$ we recover the affine action that we
introduced above.

\subsection{The affine root system}
The hyperplane of fixed points of $s_{\alpha,n}$ in $\hV^\ast$ is 
$$
\hH_{\alpha,n}:=\{(v,m)\mid
\langle\alpha,v\rangle=mn\}.
$$
Set $\hX:=X\oplus\DZ$ and consider
$\hX$  as a subset in $\hV$ by the obvious
embedding. We denote by $\delta$ the element $(0,-1)\in\hX$. For $\alpha\in R$, $n\in\DZ$, the {\em
  affine root} $\alpha+n\delta\in\hX$ is then an equation of
$\hH_{\alpha,n}$. The set of {\em (real) affine roots} is, by definition,
$$
\hR:=\{\alpha+n\delta\mid \alpha\in R,n\in\DZ\}\subset\hX.
$$

Now let us choose a system $R^+\subset R$ of positive roots. Set
$$
\hR^+:=\{\alpha+n\delta\mid \alpha\in R, n>0\}\cup
\{\alpha+n\delta\mid \alpha\in R^+, n\ge 0\}.
$$
Then $\hR=\hR^+\dot\cup -\hR^+$. Let us denote by $\Pi\subset R^+$
the subset of simple roots, and by $\gamma\in R^+$ the largest
root. The set of {\em simple affine roots} is 
$$
\widehat \Pi:=\Pi\cup \{-\gamma+\delta\}.
$$
Note that each $\beta\in\hR^+$ has a unique expression as a sum of
elements in $\widehat\Pi$. 
Set
$$
\hCS:=\{s_{\alpha,0}\mid \alpha\in \Pi\}\cup\{s_{\gamma,1}\}.
$$
Then $\hCS$ is a set of generators of $\hCW$, and $(\hCW,\hCS)$ is a
 Coxeter system. This means that we have a corresponding length
function $l\colon \hCW\to \DN$ and a Bruhat order ``$\varle$'' on $\hCW$.
We can now associate the following moment graph to our data.
\begin{definition}\label{def-affmom} The {\em affine moment graph}
  $\hCG=\hCG(\hR,\hR^+)$ is given by the following data:
\begin{enumerate}
\item The underlying lattice is $\hX$.
\item The set of vertices is $\CV:=\hCW$.
\item The vertices $x,y\in\hCW$ are connected by a
(single) edge $E$ if there are $\alpha\in R$ and $n\in\DZ$ with
$x=s_{\alpha,n}y$. The edge $E$ is then labelled by the unique positive root in
$\{\alpha+n\delta,-\alpha-n\delta\}$.
\item The order ``$\varle$'' is the
affine Bruhat order.
\end{enumerate}
\end{definition}

Let us now return to the case of arbitrary moment graphs.
\subsection{Base fields and the GKM-property}
Let $Y$ be a lattice and $\CG$ a moment graph over $Y$. Let $K$ be  a field of characteristic $\ne 2$ and let
$$
Y_K := Y\otimes_\DZ K 
$$ 
be the  $K$-vector space associated to the lattice $Y$. We will often denote by $\lambda$ also the element
$\lambda\otimes 1\in Y_K$. 

\begin{definition}\begin{enumerate}
\item Let $\CI$ be a subset of $\CV$. We say that
  $(K,\CI)$ is a {\em GKM-pair} if the labels on any pair of edges of
  $\CG_{\CI}$ that share a common vertex are linearly independent in
  $Y_K$, i.e.~ if for all  $x,y,y^\prime\in\CI$, $y\ne y^\prime$, and edges $E\colon x\linie y$ and
  $E^\prime\colon x\linie y^\prime$ we have
  $\alpha(E)\not\in K\alpha(E^\prime)$. 
\item If $w\in\CV$, then we say that $(K,w)$ is a GKM-pair if
  $(K,\{\le w\})$ is a GKM-pair. 
\end{enumerate}
\end{definition}

Denote by $S:=S_K(Y_K)$ the symmetric algebra of the $K$-vector
space $Y_K$.
We consider $S$ as a $\DZ$-graded algebra. The grading is determined
by setting $\deg \lambda:=2$ for all non-zero $\lambda$ in
$Y_K$. Almost all $S$-modules in the following will be assumed to be
$\DZ$-graded. 

For a $\DZ$-graded space $M$ and $n\in \DZ$ we denote by $M_{\{n\}}$
its homogeneous component of degree $n$.  In the following, almost all maps  $f\colon M\to N$
of $\DZ$-graded spaces will be of degree zero, i.e.~they satisfy
$f(M_{\{n\}})\subset N_{\{n\}}$ for all $n\in \DZ$.  For $l\in\DZ$ we denote by
$M\{l\}$ the graded space obtained from $M$ by shifting the grading  in such a way that $M\{l\}_{\{n\}}=M_{\{l+n\}}$.

\subsection{Sheaves on moment graphs}
Now we come to one of the central definitions for our theory.

\begin{definition}
A {\em $K$-sheaf} $\SM$ on the moment graph $\CG$ is given by the data
$\left(\{\SM^x\},\{\SM^E\},\{\rho_{x,E}\}\right)$, where 
\begin{enumerate}
\item $\SM^x$ is an $S$-module for any vertex $x$,
\item $\SM^{E}$ is an $S$-module with $\alpha(E) \SM^E=0$ for
  any edge $E$,
\item $\rho_{x,E}\colon\SM^x\to\SM^E$ is a homomorphism of
  $S$-modules for any vertex $x$ lying on the edge $E$.
\end{enumerate}
\end{definition}
The {\em support} of a sheaf $\SM$ is defined as 
$\supp\, \SM:=\{x\in \CV\mid \SM^x\ne 0\}$.
For $x\in\CV$ we call the space
$\SM^x$ the {\em stalk} of $\SM$ at $x$. Let us define $\CE^x\subset
\CE$ as the set of edges that contain the vertex $x$, i.e.~
$$
\CE^x:=\{E\in\CE\mid E\colon x\linie y\text{ for some $y\in\CV$}\}.
$$
Then we define   the {\em
  costalk} of $\SM$ at $x$ as 
\begin{align*}
\SM_x&:=\{m\in \SM^x\mid \rho_{x,E}(m)=0\text{ for all $E\in\CE^x$}\}.
\end{align*}
By definition, $\SM_x$ is a sub-$S$-module of $\SM^x$. Later we also need
the following intermediate module. For $x\in\CV$ let $\CE^{\delta
  x}\subset\CE^x$ be the set of all directed edges {\em starting at
  $x$}, i.e.~ 
$$
\CE^{\delta x}=\{E\in \CE\mid E\colon x\to y\text{ for some $y\in\CV$}\}.
$$
Recall that the notation $E\colon x\to y$ implies $x<y$. We set
$$
\SM_{[x]}:=\{m\in \SM^x\mid \rho_{x,E}(m)=0\text{ for all
  $E\in\CE^{\delta x}$}\}.
$$
Then we have inclusions $\SM_x\subset\SM_{[x]}\subset\SM^x$.  

\subsection{Sections of sheaves}

The {\em space of global sections} of a sheaf $\SM$ is by definition
the space
$$
\Gamma(\SM):=\left.\left\{(m_x)\in\prod_{x\in\CV}\SM^x\right|
\,\begin{matrix}
\rho_{x,E}(m_x)=\rho_{y,E}(m_y) \\
\text{ for all edges $E\colon x\linie y$}
\end{matrix}
\right\}.
$$
More generally, for a subset  $\CI$  of $\CV$ we define the {\em
  sections over $\CI$} by
$$
\Gamma(\CI,\SM):=\left.\left\{(m_x)\in\prod_{x\in\CI}\SM^x\right|
\,\begin{matrix}
\rho_{x,E}(m_x)=\rho_{y,E}(m_y) \\
\text{ for $E\colon x\linie y$ with $x,y\in\CI$}
\end{matrix}
\right\}.
$$
Note that for each pair $\CI^\prime\subset\CI$ there is an obvious
restriction map $\Gamma(\CI,\SM)\to \Gamma(\CI^\prime,\SM)$.

\subsection{Braden-MacPherson sheaves}

Now we introduce the sheaves $\SB(w)$, for $w\in\CV$, that
are most important for our approach. They made their first appearance in the article \cite{BMP},
where they were called the {\em canonical sheaves}. 

The construction of $\SB(w)$
is easily motivated by the following problem. Suppose that we want to
find  a global section $m=(m_x)$ of a sheaf $\SM$ on $\CG$. We
could try to construct $m_x$  vertex by vertex following the partial
order (note that so far none of our definitions used the partial
order), i.e.~suppose that we are given an
element $m_y\in\SM^y$ for all $y>x$ such that
$(m_y)\in\Gamma(\{>x\},\SM)$. Then we want to find an extension to the
vertex $x$, i.e.~ an element $m_x\in\SM^x$ such that
$((m_y),m_x)\in\Gamma(\{\ge x\},\SM)$. This means that $m_x$ should
have the following property: For each edge $E$ starting at $x$ and
ending at some $y>x$ we have $\rho_{x,E}(m_x)=\rho_{y,E}(m_y)$. This immediately leads to the following definitions. 

Fix  $x\in \CV$ and recall that we defined $\CE^{\delta x}\subset\CE$
as the set of all edges that start at $x$.
Let   
$$
\CV^{\delta x}:=\{y\in\CV\mid \text{ there is $E\colon x\to
  y\in\CE^{\delta x}$}\}
$$
be the set of the corresponding endpoints. For a sheaf $\SM$ on $\CG$  define $\SM^{\delta x}\subset\bigoplus_{E\in\CE^{\delta x}}\SM^E$ as the image of  the composition
$$
\Gamma(\{ >x\},\SM)\subset\bigoplus_{y>x}\SM^y \stackrel{p}\to\bigoplus_{y\in\CV^{\delta x}} \SM^y\stackrel{\rho}\to\bigoplus_{E\in\CE^{\delta x}}\SM^E,
$$
where $p$ is the projection along the decomposition and
$\rho=\bigoplus_{E\colon x\to y}\rho_{y,E}$. 
Moreover, we  define  
$$
\rho_{x,\delta x}:=(\rho_{x,E})_{E\in\CE^{\delta x}}^T\colon \SM^x\to\bigoplus_{E\in\CE^{\delta x}}\SM^E.
$$

The following statement follows directly from the definitions.
\begin{lemma} For each sheaf $\SM$ on $\CG$ the canonical restriction
  map $\Gamma(\{\ge x\},\SM)\to\Gamma(\{>x\},\SM)$ is surjective if
  and only if the set $\SM^{\delta x}$ is contained in the image of
  the map $\rho_{x,\delta x}$. 
\end{lemma}

Let us now suppose that for any $w\in\CV$ the set $\{\le w\}$ is finite.
The following theorem introduces the {\em Braden--MacPherson sheaves}
on $\CG$.
\begin{theorem} \label{theorem-BMPsheaf}  For each $w\in \CV$ there is an up to isomorphism
  unique $K$-sheaf $\SB(w)$ on $\CG$ with the following properties:
\begin{enumerate}
\item The support of $\SB(w)$ is contained in $\{x\in\CV\mid x\le
  w\}$, and $\SB(w)^w\cong S$.
\item If $E\colon x\stackrel{\alpha}\to y$ is a directed edge, then the map $\rho_{y,E}\colon \SB(w)^y\to\SB(w)^E$ is surjective with kernel $\alpha\SB(w)^y$.
\item For any $x\in\CV$, the image of $\rho_{x,\delta x}$ is
  $\SB(w)^{\delta x}$, and  $\rho_{x,\delta x}\colon
  \SB(w)^x\to\SB(w)^{\delta x}$ is a projective cover in the category
  of graded $S$-modules. 
\end{enumerate}
\end{theorem}
By a {\em projective cover} of a graded $S$-module $M$ we mean a
graded free $S$-module $P$ together with a surjective homomorphism  $f\colon
P\to M$ of graded $S$-modules such that for any homomorphism $g\colon
Q\to P$ of graded $S$-modules the following holds: If $f\circ g$ is
surjective, then $g$ is surjective.

There is, in fact, nothing to prove for the above  theorem, as it describes
an inductive procedure for the construction of $\SB(w)$ which
makes its isomorphism class unique, cf.~\cite{BMP,FieAdv}. 

Next we come to the first appearance of the GKM-property. Recall that
for $x,w\in\CV$ with $x\le w$ we have natural inclusions $\SB(w)_x\subset\SB(w)_{[x]}\subset\SB(w)^x$.

\begin{lemma} \label{lemma-costalk} Choose  $x,w\in\CV$ with $x\le
  w$.  Let
  $\alpha_1$,\dots,$\alpha_n$ be the labels of all edges of 
  $\CG_{\le w}$ that end at $x$, and let $\beta_1$,\dots,$\beta_m$ be
  the labels that start at $x$. Then the following holds.
\begin{enumerate}
\item We have
  $\alpha_1\cdots\alpha_n\beta_1\cdots\beta_m\SB(w)^x\subset\SB(w)_x$.
\item Suppose that $(K,w)$ is a GKM-pair. Then 
$
\SB(w)_{x}=\alpha_1\cdots\alpha_n\SB(w)_{[x]}$.
\end{enumerate}
\end{lemma}
\begin{proof} If
  $m\in\alpha_1\cdots\alpha_n\beta_1\cdots\beta_m\SB(w)^x$, then
  $\rho_{x,E}(m)=0$ for all edges $E$ that contain $x$ as a vertex,
  since $\SB(w)^E$ is annihilated by one of the $\alpha$'s or one of
  the 
  $\beta$'s. Hence (1). Now suppose that $(K,w)$ is a GKM-pair and let $m\in\alpha_1\cdots\alpha_n\cdot\SB(w)_{[x]}$. If
  $E$ is an edge that starts at $x$, then $\rho_{x,E}(m)=0$ by
  definition of $\SB(w)_{[x]}$. If $E$ ends at $x$, then $E$ is labelled by one of the
  $\alpha_i$, so $\rho_{x,E}(m)=0$. Hence
  $\alpha_1\cdots\alpha_n\cdot\SB(w)_{[x]}\subset\SB(w)_{x}$. 

Clearly, $\SB(w)_{x}\subset\SB(w)_{[x]}$. If $m\in \SB(w)_x$, then
$\rho_{x,E}(m)=0$ for each edge $E$ ending at $x$. Since for such $E$
we can find an isomorphism $\SB(w)^E\cong\SB(w)^x/\alpha(E)\SB(w)^x$
in such a way that $\rho_{x,E}\colon\SB(w)^x\to\SB(w)^E$ identifies
with the canonical quotient map, $m$ must be divisible by $\alpha(E)$
in $\SB(w)^x$. Now $\alpha(E)$ acts injectively on
$\SB(w)^{\delta x}$ (since, by the GKM-property, it acts injectively
on $\bigoplus_{E\in\CE^{\delta x}}\SB^E$) and we deduce $\alpha(E)^{-1}m \in \SB(w)_{[x]}$. Hence (again by the GKM-property) $\alpha_1^{-1}\cdots\alpha_n^{-1}m \in \SB(w)_{[x]}$, and we deduce $\SB(w)_x\subset\alpha_1\cdots\alpha_n\cdot\SB(w)_{[x]}$.
\end{proof}

We now return to the case of affine moment graphs.

\subsection{The affine Hecke algebra}
Let
$\hbH=\bigoplus_{x\in\hCW} \DZ[v,v^{-1}] T_x$ be the Hecke algebra
associated to the Coxeter system $(\hCW,\hCS)$. Its multiplication is determined by the formulas
\begin{eqnarray*}
{ T}_x\cdot { T}_{y} & = & { T}_{xy}\quad\text{if $l(xy)=l(x)+l(y)$}, \\
{ T}_s^2 & = & v^{-2}T_e+(v^{-2}-1)T_s \quad\text{for $s\in\hCS$}.
\end{eqnarray*}
Then ${ T}_e$ is a unit in $\hbH$ and for any $x\in\hCW$ there exists an inverse of ${ T}_x$ in $\hbH$. For $s\in\hCS$ we have ${ T}_s^{-1}=v^2T_s+(v^2-1)$. There is a duality (i.e.~a $\DZ$-linear anti-involution) $d\colon\hbH\to\hbH$, given by $d(v)=v^{-1}$ and $d(T_x)  =  T_{x^{-1}}^{-1}$ for $x\in\hCW$. 

Set $H_x:=v^{l(x)}T_x$. Recall the following result:

\begin{theorem}[\cite{MR560412,MR1445511}]\label{self-dual elts} 
For any $w\in\hCW$ there exists a unique element
$\ul{H}_w=\sum_{x\in\hCW} h_{x,w}(v) H_x\in\hbH$ with the following properties:
\begin{enumerate}
\item \label{prop of H: duality} $\ul{H}_w$ is self-dual, i.e.~$d(\ul{H}_w)=\ul{H}_w$. 
\item \label{prop of H: support} $h_{x,w}(v)=0$ if $x\not\leq w$, and $h_{w,w}(v)=1$,
\item \label{prop of H: norm} $h_{x,w}(v)\in v\DZ[v]$ for $x<w$.
\end{enumerate}
\end{theorem} 
For example, we have $\ul{H}_e=H_e$ and $\ul{H}_s=H_s+v H_e$
for each $s\in\hCS$. The polynomials $h_{x,w}$ are called the {\em
  affine Kazhdan--Lusztig polynomials}.

\subsection{The main conjecture}\label{sec-mainconj}

Now let $\hCG$ be the  moment graph that we associated to the
affine root system $\hR$ in Definition  \ref{def-affmom}. Let $w\in\CV=\hCW$.
By the defining properties of $\SB(w)$ listed in Theorem
\ref{theorem-BMPsheaf},  the
$S$-module $\SB(w)^x$ is graded free for all $x\le w$. For any graded
free $S$-module $M$ we set
$$
\rk^\prime\,M:=v^{l_1}+\dots+v^{l_n},
$$
where $l_1,\dots,l_n\in\DZ$ are such that $M\cong\bigoplus_{i=1}^n
S\{l_i\}$.

\begin{definition} The {\em graded character of $\SB(w)$} is 
$$
h(\SB(w)):=\sum_{x\le w}\rk^\prime\,\SB(w)^x v^{l(w)}T_x\in\hbH.
$$
\end{definition}

Here is the main conjecture.

\begin{conjecture}\label{conj-KL} If $(K,w)$ is a GKM-pair,  then
$$
h(\SB(w))=\ul{H}_w.
$$ 
\end{conjecture}

In the following section we shortly present the main application of
the above conjecture. 
 
\subsection{Lusztig's conjecture}\label{sec-Lconj}

Suppose that $k$ is a field of positive non-even characteristic. Let $\ol k$
be an algebraic closure of $k$ and let $G^\vee_{\ol k}$ be the  connected,
simply connected simple algebraic group over $\ol k$ whose root system is dual to $R$.
 Let $\widehat w_0\in \widehat \CW$ be the element that corresponds to
 the lowest alcove in the anti-fundamental box (cf.~\cite{FieModRep}).  If $\ch k$ is at least the Coxeter number of our root system $R$, then the pair $(k,w)$ is a GKM-pair for each $w\varle \widehat w_0$. 

The main result in \cite{FieModRep} and \cite{FieCharIrr} is that
Conjecture \ref{conj-KL} implies Lusztig's conjecture on the
characters of irreducible rational representations of $G_{\ol k}^\vee$. More precisely:

\begin{theorem} Suppose that $\ch k$ is at least the Coxeter number of
  $R$. If Conjecture $\ref{conj-KL}$ holds for all $w\varle \widehat
  w_0$, then Lusztig's conjecture  (cf.~\cite{Lus}) on the characters of irreducible rational representations of $G_{\ol k}^\vee$ holds.
\end{theorem}

\section{$\CZ$-modules} 
From now on we fix the root system $R$, a system $R^+\subset R$ of positive
roots  and the field $K$. We denote by
$\hCG$ the moment graph associated to these data.
In the following  we want to translate Conjecture \ref{conj-KL} into a
similar character conjecture for certain modules over a commutative
and associative algebra. 

\subsection{The structure algebra}\label{subsec-strucalg}

The {\em structure algebra} of $\hCG$ over $K$ is, by definition, the algebra
$$
\CZ := \left\{ (z_x)\in \prod_{x\in \hCW}S\left| \, 
\begin{matrix} 
 z_x\equiv z_{s_{\alpha,n}x} \mod \alpha+n\delta \\
 \text{ for all $x\in \hCW$, $\alpha\in R$, $n\in\DZ$}
\end{matrix}
\right.\right\}.
$$
For a  subset $\CI$ of $\hCW$ we define the {\em local structure algebra} 
$$
\CZ(\CI):=\left\{ (z_x)\in \prod_{x\in \CI}S\left| \, 
\begin{matrix} 
 z_x\equiv z_{s_{\alpha,n}x} \mod \alpha+n\delta \\
 \text{ for all $x\in \CI$, $\alpha\in R$, $n\in\DZ$  with $s_{\alpha,n}x\in\CI$}
\end{matrix}
\right.\right\}.
$$

The algebras $\CZ$ and $\CZ(\CI)$ are the global and local sections of the
  {\em structure sheaf} $\SZ$ that is defined by $\SZ^x=S$,
  $\SZ^E=S/\alpha(E)S$, and $\rho_{x,E}\colon S\to S/\alpha(E)S$ the
  canonical map.
There are obvious restriction maps  $\CZ(\CI)\to\CZ(\CI^\prime)$
whenever $\CI^\prime\subset\CI$ are subsets of $\hCW$. Both $\CZ=\CZ(\hCW)$ and
$\CZ(\CI)$ are commutative, associative, unital $S$-algebras with
coordinatewise addition and multiplication. More generally, coordinatewise multiplication makes $\Gamma(\CI,\SM)$ into a
$\CZ(\CI)$-module (hence a  $\CZ$-module) for each sheaf $\SM$ on $\hCG$.

If $\CI$ is infinite, then $\CZ(\CI)$ and $\prod_{x\in\CI} S$ are {\em not} 
graded algebras in any canonical sense. However, we can define the
graded component of $\CZ(\CI)$ of degree $n$ as
the intersection $\CZ(\CI)\cap \prod_{x\in\CI} S_{\{n\}}$. It then makes sense to talk about $\DZ$-graded $\CZ(\CI)$-modules, i.e.~ modules that carry a $\DZ$-grading such that each {\em homogeneous} element of $\CZ(\CI)$ acts homogeneously.

\begin{definition} For $w\in\hCW$ we denote by $\CB(w):=\Gamma(\SB(w))$ the $\CZ$-module of global sections of $\SB(w)$.
\end{definition}
Since $\SB(w)$ is supported on $\{\le w\}$, the action of $\CZ$ on
$\CB(w)$ factors over an action of  $\CZ(\{\le w\})$ on $\CB(w)$. 

Our next objective is to translate Conjecture \ref{conj-KL} into an
analogous multiplicity conjecture for the special modules $\CB(w)$. In
order to do this we have to recover the stalks $\SB(w)^x$ in terms of $\CB(w)$. 

\subsection{The generic decomposition}\label{subsec-gendecomp}

Fix  a field $K$ and a {\em finite} subset $\CI$ of $\hCW$ such that $(K,\CI)$ is a GKM-pair. In particular,
we have $\alpha(E)\ne 0$ in $S$ for all edges $E$ of $\hCG_{\CI}$. Let us denote
by $Q_\CI$ the ring obtained by adjoining $\alpha(E)^{-1}$ for
all those $E$ to $S$
inside the quotient field $\Quot(S)$. For an $S$-module $M$ we write
$M_{Q_\CI}:=M\otimes_S Q_\CI$ for the $Q_\CI$-module obtained by extending
scalars. In particular, we have a $Q_\CI$-algebra $\CZ(\CI)_{Q_\CI}$. The
natural inclusion $\CZ(\CI)\subset\bigoplus_{x\in\CI} S$ induces an
inclusion $\CZ(\CI)_{Q_\CI}\subset\prod_{x\in\CI} Q_\CI$.

\begin{lemma}[\cite{FieAdv}] Suppose that $(K,\CI)$ is a GKM-pair and that $\CI$ is finite. Then the following holds:
\begin{enumerate}
\item  The canonical inclusion $\CZ(\CI)_{Q_\CI}\subset\bigoplus_{x\in\CI} Q_\CI$
is a bijection.
\item If $\CM$ is a $\CZ(\CI)$-module, then there is a canonical decomposition $\CM_{Q_\CI}=\bigoplus_{x\in\CI} \CM_{Q_\CI}^x$ such that $(z_x)\in \CZ(\CI)_{Q_\CI}$ acts on $\CM_{Q_\CI}^x$ as multiplication by $z_x$.
\end{enumerate}
\end{lemma}

 Let $\CM$ be a $\CZ(\CI)$-module that is torsion free as an
 $S$-module. Then we have a canonical inclusion $\CM\subset\CM_{Q_\CI}$. For
 a subset  $\CJ$ of $\CI$ let us define 
$$
\CM^\CJ:=\im\left(\CM\subset \CM_{Q_\CI}=\bigoplus_{x\in\CI}\CM_{Q_\CI}^x\to  \bigoplus_{x\in \CJ}\CM_{Q_\CI}^x\right)
$$ 
and 
$$
\CM_\CJ:=\CM\cap\bigoplus_{x\in \CJ}\CM_{Q_\CI}^x.
$$

If $\CJ^\prime$ is a subset of $\CJ$, then we have a natural surjection
$\CM^\CJ\to\CM^{\CJ^\prime}$ and a natural injection
$\CM_{\CJ^\prime}\to\CM_{\CJ}$. We define the {\em stalk} of $\CM$ at $x$ as $\CM^x:=\CM^{\{x\}}$ and
the {\em costalk} of $\CM$ at $x$ as $\CM_x:=\CM_{\{x\}}$.
For $x\le w$ we set $\CM^{\ge x}=\CM^{\{\ge x\}}$ and $\CM^{>
  x}=\CM^{\{> x\}}$.  We let $\CM_{[x]}$ be the kernel of the
surjection $\CM^{\ge x}\to\CM^{>x}$. Then we have inclusions 
$$
\CM_x\subset\CM_{[x]}\subset\CM^x.
$$

The following lemma is a consequence of the definitions.

\begin{lemma}\label{lemma-stalks} The costalk $\CM_x$  is the biggest submodule of
  $\CM$ on which $(z_y)\in\CZ(\CI)$ acts as multiplication with
  $z_x$. Analogously, the stalk $\CM^x$ is the biggest quotient of
  $\CZ$ on which $(z_y)$ acts as multiplication with $z_x$.
\end{lemma}

Now we can compare the stalks and costalks of $\CB(w)$ and of $\SB(w)$.
\begin{proposition} \label{prop-stgsec} The canonical map
  $\CB(w)=\Gamma(\SB(w))\to \SB(w)^x$ induces an isomorphism 
$$
\CB(w)^x\stackrel{\sim}\to \SB(w)^x.
$$ 
This isomorphism restricts to the following isomorphisms on
subspaces:
$$
\CB(w)_x \stackrel{\sim}\to \SB(w)_x,\quad
\CB(w)_{[x]} \stackrel{\sim}\to \SB(w)_{[x]}.
$$
In particular, the map $\SB(w)_x\to \SB(w)^x$ identifies with the map $\CB(w)_x\to \CB(w)^x$. 
\end{proposition}
\begin{proof} By the characterization in the previous lemma and the
  freeness of the stalks $\SB(w)^y$ it is
  clear that we have an embedding $\CB(w)^x\subset\SB(w)^x$. From the
  construction of $\SB(w)$ it follows that each element in $\SB(w)^x$
  appears as a component of a global section, hence the above
  embedding is also a surjection. More generally, since each local section
  of $\SB(w)$ extends to a global section we have induced isomorphisms $\CB(w)^{\ge x}\cong\Gamma(\{\ge
  x\},\SB(w))$ and $\CB(w)^{> x}\cong\Gamma(\{>
  x\},\SB(w))$. From the definition it is clear that $\SB(w)_{[x]}$ is
  the kernel of the restriction map $\Gamma(\{\ge
  x\},\SB(w))\to\Gamma(\{>x\},\SB(w))$, hence we get an induced
  isomorphism $\CB(w)_{[x]}\cong\SB(w)_{[x]}$. Finally,
  $\CB(w)_x\cong\SB(w)_x$ follows directly from the previous lemma.
\end{proof}
Hence, in order to prove Conjecture \ref{conj-KL} we can calculate the
graded character of the stalks $\CB(w)^x$ for all pairs $(x,w)$.  The
$\CB(w)$ have the advantage that they are modules over a commutative
algebra. Moreover, they can be constructed by an alternative method, namely as
direct summands of the {\em Bott--Samelson modules}.

\subsection{Bott--Samelson modules}

Fix a simple affine reflection $s\in\hCS$ and define the sub-$S$-algebra
$$
\CZ^s:= \left\{(z_x)\in \CZ\mid z_x=z_{xs}\text{ for all $x\in \hCW$}\right\}.
$$

\begin{definition} The translation functor associated to $s$ is the functor $\vartheta_s$ that maps a $\CZ$-module $\CM$ to the $\CZ$-module $\CZ\otimes_{\CZ^s} \CM$, and a map $f\colon \CM\to \CN$ of $\CZ$-modules to the map $1\otimes f\colon \CZ\otimes_{\CZ^s} \CM\to \CZ\otimes_{\CZ^s} \CN$. 
\end{definition}
Let $\CM(e)$ be the $\CZ$-module that is free of rank one over $S$ and
on which $(z_x)\in\CZ$ acts by multiplication with $z_e$.

\begin{definition} Let $\bs=(s_1,\dots, s_l)$ be a sequence in $\hCS$. The module $\CB(\bs):=\vartheta_{s_l}\cdots \vartheta_{s_1}(\CM(e))$ is called the {\em Bott--Samelson module} associated to $\bs$. 
\end{definition}

The following theorem shows that it is possible to construct the
global sections $\CB(w)$ of the Braden-MacPherson sheaf $\SB(w)$
directly in the category of $\CZ$-modules, i.e.~without refering to
sheaves on a moment graph. Its  proof  is
contained in \cite{FieTAMS} (cf.~ the proof of Theorem 6.1 and
Corollary 6.5 in {\em loc.~cit.}).

\begin{theorem}\label{theorem-BSspec} Fix $w\in\hCW$ and suppose that $(K,w)$ is a
  GKM-pair. Let $\bs=(s_1,\dots, s_l)$ be a sequence in $\hCS$ such that $w=s_1\cdots s_l$ is a reduced expression. Then there are $x_1,\dots,x_n<w$ and $l_1,\dots, l_n\in \DZ$ such that
  $$
  \CB(\bs)\cong\CB(w)\oplus\bigoplus_{i=1}^n \CB(x_i)\{l_i\}.
  $$
\end{theorem}

\subsection{A duality}\label{sec-duality}
For  graded $S$-modules $M$ and $N$ let
$$
\Hom^n_S(M,N):=\Hom_S(M,N\{n\})
$$ 
be the space of degree $n$
homomorphisms between $M$ and $N$, and let
$$
\dual M:=\bigoplus_{n\in\DZ}\Hom^n_S(M,S\{n\})
$$
be the dual of $M$. This is a graded $S$-module as well, and we have
$\dual S\cong S$ and $\dual(M\{l\})\cong(\dual M)\{-l\}$ as $S$-modules. 

Each $\CZ$-module is naturally an $S$-module as $\CZ$ is a unital $S$-algebra. If $\CM$ is a $\CZ$-module, then $\dual \CM$ is a $\CZ$-module as
well with the action  given by 
$$
(z.f)(m):=f(z.m)
$$
for all $z\in\CZ$, $f\in\dual \CM$ and $m\in\CM$. 

\begin{theorem}[{\cite[Cor.~5.7, Thm.~6.1, Prop.~7.1]{FieTAMS}}] \label{theorem-duality}
\begin{enumerate}
\item For each sequence $\bs=(s_1,\dots,s_l)$ of length
  $l$ the Bott--Samelson module
  $\CB(\bs)$ is self-dual up to a shift by $2l$, i.e.~ we have
$$
\dual\CB(\bs)\cong \CB(\bs)\{2l\}.
$$ 
This isomorphism induces isomorphisms
$\dual(\CB(\bs)^x)\cong\CB(\bs)_x\{2l\}$ for all $x\in\hCW$.
\item For each $w\in\hCW$ the $\CZ$-module $\CB(w)$ is
  self--dual up to a shift by $2l(w)$, i.e.~ we have
$$
\dual \CB(w)\cong\CB(w)\{2l(w)\}.
$$
This isomorphism induces isomorphisms
$\dual(\CB(w)^x)\cong\CB(w)_x\{2l(w)\}$ for all $x\in\hCW$. 
\end{enumerate}
\end{theorem}

Let $\bs=(s_1,\dots,s_l)$ be a sequence in $\hCS$, and let
$J(\bs)\subset\hCW$ be the set of all elements in $\hCW$ that are
smaller than or equal to some
subword of $s_1\cdots s_l$. Let $K$ be a
field.
\begin{definition} We say that $(K,\bs)$ is a GKM-pair if $(K,J(\bs))$ is a
GKM-pair. 
\end{definition}
If $(K,w)$ is a GKM-pair, then the Bott--Samelson module $\CB(\bs)$ decomposes
into a direct sum of shifted copies of various $\CB(w)$'s with $w\in
J(\bs)$. This follows from Theorem \ref{theorem-BSspec}. Theorem \ref{theorem-duality}, Lemma \ref{lemma-costalk}, Proposition
\ref{prop-stgsec} and the fact that the number of edges of $\hCG_{\le
  w}$ that end at $x$ is $l(x)$, yield  the following:

\begin{corollary}\label{cor-specmod} 
\begin{enumerate} Suppose that $(K,\bs)$ is a GKM-pair. Let $x\in\hCW$.
\begin{enumerate}
\item The $S$-modules $\CB(\bs)^x$, $\CB(\bs)_x$ and
  $\CB(\bs)_{[x]}$ are graded free of finite rank.
\item  Let the numbers $a_1,\dots, a_r$ be defined by
$\CB(\bs)^x\cong\bigoplus_{i=1,\dots r} S\{a_i\}$. 
Then we have
\begin{align*}
\CB(\bs)_{x}&\cong\bigoplus_{i=1,\dots,r} S\{-a_i-2l\}, \\
\CB(\bs)_{[x]}&\cong\bigoplus_{i=1,\dots,r} S\{-a_i-2l+2l(x)\}.
\end{align*}
\end{enumerate}
\item
Let $w\in\hCW$ and suppose that $(K,w)$ is a GKM-pair. Let $x\in\hCW$.
\begin{enumerate}
\item
The $S$-modules $\CB(w)^x$, $\CB(w)_x$ and $\CB(w)_{[x]}$ are graded  free
  of finite rank. 
\item  Let the numbers $b_1,\dots,b_r$  be defined by $
\CB(w)^x\cong\bigoplus_{i=1,\dots,r} S\{b_i\}$.
Then we have 
\begin{align*}
\CB(w)_{x}&\cong\bigoplus_{i=1,\dots,r} S\{-b_i-2l(w)\}, \\
\CB(w)_{[x]}&\cong\bigoplus_{i=1,\dots,r} S\{-b_i-2l(w)+2l(x)\}.
\end{align*}
\end{enumerate}
\end{enumerate}
\end{corollary}
\begin{remark} For a graded free $S$-module $M\cong
  \bigoplus_{i=1,\dots r} S\{l_i\}$ we set $\rk\, M:=
  v^{-l_i}\in\DZ[v,v^{-1}]$ (note that in the definiton of
  $\rk^\prime$ in Section \ref{sec-mainconj} we used the opposite sign
  convention). From the corollary above we deduce
$$
\rk^\prime\,\CB(w)^{x}=\rk\, \CB(w)_{[x]}v^{2l(x)-2l(w)},
$$
and, together with Proposition \ref{prop-stgsec}, we get 
$$
h(\SB(w))=\sum_{x\le w} \rk\, \CB(w)_{[x]}v^{2l(x)-l(w)} T_x.
$$
We see that our Main Conjecture \ref{conj-KL} coincides with
Conjecture 8.3 of \cite{FieTAMS} and Vermutung 1.3 in \cite{SoeJIMJ}
in the cases discussed there. (Note that $\tilde T_x=v^{l(x)} T_x$ and
that the $B(w)$ in \cite{FieTAMS} are shifted by $l(w)$ with respect
to our $\CB(w)$, by Corollary 6.5 in \cite{FieTAMS}). When we stated
Conjecture \ref{conj-KL} we did not know  that $\SB(w)_{[x]}$ is a
graded free $S$-module, which is why we used the graded character of the
stalks.
\end{remark}

Now we can state the following conditions that are equivalent to our
main conjecture.

\begin{proposition}\label{prop-sme} Suppose that $(K,w)$ is a
  GKM-pair. Then Conjecture \ref{conj-KL} is equivalent to each of the
  following statements:
\begin{enumerate}
\item For each $x< w$ the $S$-module $\CB(w)^x$ is
  generated in degrees $<l(w)-l(x)$.
\item For each $x<w$ the $S$-module $\CB(w)_{[x]}$ lives in degrees
  $>l(w)-l(x)$.
  \item For each $x<w$ the $S$-module $\CB(w)_{x}$ lives in degrees $>l(w)+l(x)$. 
\end{enumerate}
\end{proposition}
\begin{proof} The equivalence between the three statements above is an
  immediate consequence of Corollary \ref{cor-specmod}. That the
  Conjecture \ref{conj-KL} and  the statement (2) are equivalent is
  proven in \cite[Proposition 8.4]{FieTAMS}.
\end{proof}

Finally we   associate
the following character to $\CB(\bs)$:
\begin{align*}
h(\CB(\bs))& :=\sum_{x\in\hCW} \rk^\prime\, \CB(\bs)^{x} v^{l}T_x\\
&= \sum_{x\in\hCW} \rk\, \CB(\bs)_{[x]} v^{2l(x)-l}T_x
\end{align*}
Up to a factor of $v^{-l}$ on the right hand side this coincides with
the definition of $h_{\le,l}$ in Section 4.3 of \cite{FieModRep}. The
Lemma 4.4. in {\em loc.~cit.} yields the following result.
\begin{lemma}\label{lemma-rankind}
 For a sequence $\bs=(s_1,\dots,s_l)$ in $\hCS$  we have 
$$
h(\CB(\bs)) = \ul H_{s_1}\cdots \ul H_{s_l}. 
$$
In particular, the graded rank of $\CB(\bs)^x$ for fixed $x$ is
independent of the base field $K$ as long as $(K,\bs)$ is a GKM-pair. 
\end{lemma}

\section{A Lefschetz theory for sheaves on moment graphs} \label{sec-Leftheo}

We approach Conjecture \ref{conj-KL} by
studying an analogue of the {\em Hard Lefschetz property} of the
intersection cohomology of projective varieties. For this we have to
replace the symmetric algebra $S$ by a polynomial ring in one
variable.

Let $K[t]$ be the polynomial ring in one variable over the field
$K$ which we consider as a graded algebra with the grading given by setting $t$ in
degree $2$. We endow $K[t]$ with the structure of an $S$-algebra via the
homomorphism $\tau\colon S\to K[t]$ of graded
algebras that is given by $\tau(\alpha)=t$ for all $\alpha\in\widehat\Pi$.

Let $\bs=(s_1,\dots,s_l)$ be a sequence in $\hCS$ and denote by
$\hR^+_{\bs}\subset\hR^+$ the subset of all positive roots that appear
as a label on the graph $\hCG_{J(\bs)}$ (recall that
$J(\bs)\subset\hCW$ is the subset of elements that are smaller than or
equal to a subexpression of $s_1\cdots
s_l$).  Analogously, for $w\in\hCW$ let us define $\hR^+_{\le w}\subset\hR^+$ as the set of
all labels on $\hCG_{\le w}$.

We define the
{\em height} $\height(\alpha)$ of a positive affine root
$\alpha\in \hR^+$ as the number $n$ such that $\alpha$ can be written
as a sum of $n$ elements of $\widehat\Pi$ and we set
\begin{align*}
N(\bs)&:=\max_{\alpha\in\hR^+_{\bs}}\{\height(\alpha)\},\\
N(w)&:=\max_{\alpha\in\hR^+_{w}}\{\height(\alpha)\}.
\end{align*}
Clearly $N(w)=N(\bs)$ if $w=s_1\cdots s_l$ is a reduced expression for $w$.
We obviously have:
\begin{lemma} Suppose that $K$ is a field with $\ch\, K>N(\bs)$ ($\ch\, K>N(w)$). Then we have $\tau(\alpha)\ne 0$ for all $\alpha\in \hR^+_\bs$ ($\alpha\in \hR^+_w$). 
\end{lemma}
Note that $\ch\,K > N(\bs)$ ($\ch\,K > N(w)$) implies that $(K,\bs)$ (or $(K,w)$) is a GKM-pair. For any
$S$-module $M$ we let $\ol M:= M\otimes_{S} K[t]$ be the $K[t]$-module
obtained by extension of scalars. For notational convenience we write
$\ol \CB(w)$ instead of $\ol{\CB(w)}$ and $\ol\CB(w)_x$ instead of
$\ol{\CB(w)_x}$, etc.  The natural inclusion $\CB(w)_x\subset\CB(w)^x$ yields a homomorphism $i_x\colon \ol\CB(w)_x\to\ol\CB(w)^x$.

\begin{lemma}\label{lemma-inj} Let $w\in\hCW$ and suppose that $\ch\, K>N(w)$. Then the homomorphism
  $i_x\colon\ol\CB(w)_x\to \ol\CB(w)^x$ is an injective map between
  graded free $K[t]$-modules for all $x\le w$.
  It is an isomorphism if we invert the variable $t$.
\end{lemma}

\begin{proof} We have already seen that $\CB(w)^x$ and $\CB(w)_x$ are
  graded free $S$-modules.  Hence $\ol\CB(w)_x$ and $\ol \CB(w)^x$ are graded free
  $K[t]$-modules. 

Let $\alpha_1,\dots,\alpha_n$ be the labels of all edges of $\hCG_{\le
  w}$ that contain $x$ as a vertex. By Lemma \ref{lemma-costalk}, $\alpha_1\cdots
\alpha_n\SB(w)^x\subset \SB(w)_x$. Hence, if we invert the roots
$\alpha_1,\dots,\alpha_n$, then $\SB(w)_x$ and $\SB(w)^x$, and hence
$\CB(w)_x$ and $\CB(w)^x$ coincide. Since our assumptions imply that  $\tau(\alpha_i)\ne 0$ for $i=1,\dots, n$, the map $i_x\colon\ol{\CB}(w)_x\to \ol\CB(w)^x$ is an isomorphism after inverting the variable $t$. Since both spaces are free $K[t]$-modules, the map $i_x$ must be injective.
\end{proof}

\begin{proposition}\label{prop-lef} Let $w\in\hCW$ and suppose that $\ch\, K>N(w)$.  Then the following holds.
\begin{enumerate}
\item $\ol\CB(w)^w/\ol\CB(w)_w\cong K[t]/\langle t^{l(w)}\rangle$.
\item For $x<w$ let the numbers $a_i$, $b_i$, $i=1,\dots, r$ be
  defined by   
$$
\ol\CB(w)^x/\ol\CB(w)_{x}\cong\bigoplus_{i=1}^r K[t]/\langle
  t^{a_i}\rangle\{b_i\}.
$$
Then $r$ is the ungraded rank of $\CB(w)^x$
  and we have $a_i>l(x)$ for all $i=1,\dots, r$.
\end{enumerate}
\end{proposition}
\begin{proof}
 There are exactly $l(w)$ edges in $\hCG_{\le w}$ that
  contain the vertex $w$. Let $\alpha_1,\dots,\alpha_{l(w)}$ be their
  labels. Then we have $\CB(w)_w=\alpha_1\cdots\alpha_{l(w)}\CB(w)_{[w]}$
and, by definition, $\CB(w)_{[w]}=\CB(w)^w$. From this we deduce (1).

Suppose that $x<w$. From \ref{lemma-costalk}, (2) we deduce that the
minimal number of generators of $\ol\CB(w)^x/\ol\CB(w)_{x}$ is the
rank of $\CB(w)^x$. Since the number of edges of $\hCG_{\le w}$ that
end at $x$ is $l(x)$, we deduce from the same statement that there are
$\alpha_1,\dots,\alpha_{l(x)}$ such that
$\CB(w)_x=\alpha_1\cdots\alpha_{l(x)}\CB(w)_{[x]}$. So $a_i\ge
l(x)$ in any case.  

Now suppose that there are $i$ with $a_i=l(x)$. Among
those  let us choose $i$ such that  $b_i$ is maximal. Then there is an element
$m\in \CB(w)_x$ of degree $-b_i+l(x)$ that is not contained in the
subspace $\CB(w)^x_{\{< b_i\}}$ of $\CB(w)^x$ that is generated by all homogeneous components
of degree $<-b_i$. By Lemma \ref{lemma-costalk}, $m$ is of
the form $\alpha_1\cdots\alpha_{l(x)} m^\prime$ with 
$m^\prime\in\CB(w)_{[x]}$. Then $m^\prime$ is of degree $-b_i$ and is
also not contained in
$\CB(w)^x_{\{<-b_i\}}$. Hence there is a homogeneous $S$-basis of $\CB(w)^x$
that contains $m^\prime$. But $m^\prime$ lies in the kernel of the map
$\CB(w)^x\to\CB(w)^x/\CB(w)_{[x]}$, which, by construction, is a
projective cover. But then no basis element can be mapped to zero, hence we have
a contradiction. So there is no $i$ with $a_i=l(x)$.
\end{proof}

\subsection{The Lefschetz condition} \label{subsec-Lefcond}
Suppose that $A$ and $B$ are graded free $K[t]$-modules of finite
rank. Let  $f\colon A\to B$ be an injective homomorphism of
$K[t]$-modules that is an isomorphism over the ring $K[t,t^{-1}]$. So we assume that the cokernel of $f$ is a torsion $K[t]$-module, i.e.~ $t^N\coker f=0$ for $N\gg0$. Then $\coker f$ is isomorphic to a direct sum of $K[t]$-modules of the form 
$$
K[t]/\langle t^a\rangle \{b\}
$$
for some numbers $a\ge 1$ and $b\in \DZ$.

\begin{definition} We say that $f\colon A\to B$ {\em satisfies the Lefschetz condition with center $l\in \DZ$} if for each $n\ge 1$ the multiplication with $t^n$ induces an isomorphism 
$$
(\coker f)_{\{l-n\}}\stackrel{\sim}\to (\coker f)_{\{l+n\}}.
$$
\end{definition}
Our definition is motivated by 
 the Hard Lefschetz theorem in complex algebraic geometry. 
 Note that  $f\colon A\to B$ satisfies the Lefschetz condition with center $l$ if and only if 
$\coker f$ is isomorphic to a direct sum of $K[t]$-modules of the form 
$$
 K[t]/\langle t^a\rangle \{-l+a\}
$$
with $a\ge 1$.  Here is another conjecture.

\begin{conjecture}\label{conj-HL} Let $w\in \hCW$ and suppose that $\ch\, K>N(w)$. Then for all $x\le w$ the map
$$
i_x\colon\ol\CB(w)_x\to \ol\CB(w)^x
$$
satisfies the Lefschetz condition with center $l(w)$.
\end{conjecture}

 \begin{theorem} For each $w\in \hCW$ and each field $K$ with $\ch\, K>N(w)$  the  Conjecture $\ref{conj-HL}$ implies Conjecture $\ref{conj-KL}$.
\end{theorem}

\begin{proof}  Suppose that
Conjecture \ref{conj-HL} holds for the pair $(K,w)$ and let $x<w$. Then
$\CB(w)^x/\CB(w)_x$ is a direct sum of $K[t]$-modules of the form
$K[t]/\langle t^a\rangle\{-l(w)+a\}$ for various $a>0$. 
By Proposition
\ref{prop-lef} only terms  with $a>l(x)$ occur. Now
$K[t]/\langle t^a\rangle\{-l(w)+a\}$ is generated in degree
$l(w)-a$, which is strictly smaller than $l(w)-l(x)$, and we deduce that $\CB(w)^x$ is generated in
degrees $<l(w)-l(x)$. This statement is equivalent  to Conjecture
\ref{conj-KL} by Proposition \ref{prop-sme}.
\end{proof}

\subsection{The characteristic zero case}

In this section we review the topological approach of
\cite{SoeGAFA} that yields Conjecture \ref{conj-HL} in the case that
$\ch K=0$. 
\begin{theorem} \label{theorem-charzero} Suppose that $K$ is a field
  of characteristic zero. Then Conjecture $\ref{conj-HL}$ holds for all $w\in\hCW$. 
\end{theorem}
\begin{proof}[Sketch of the proof] The proof is essentially contained in the last
  section of \cite{SoeGAFA}. However,  Soergel works in the category
  of {\em special $S$-bimodules}. But in \cite{FieTAMS} it is proven that our category of
  $\CZ$-modules is equivalent to Soergel's category of
  bimodules. Moreover, in \cite{SoeGAFA} the finite dimensional
  parabolic situation is considered, while we need to work in the
  infinite dimensional Kac-Moody setting. Here the Schubert varieties
  are still finite dimensional, so there are no further technical
  difficulties to overcome. For the Kac-Moody versions of the
  well-known results about Schubert varieties see \cite{Kumar}.

  Here are Soergel's arguments: Let $G$ be the complex Kac-Moody group associated to $R$, $B\subset G$ the
  Borel subgroup associated to $R^+$ and $T\subset B$ the Kac-Moody
  torus. The first step in the proof is to identify $\CZ(\{\le w\})$ with the
  $T$-equivariant cohomology and $\CB(w)$ with the $T$-equivariant
  intersection cohomology (both with coefficients in $K$) of the
  Schubert variety $\ol{BwB/B}$ 
  associated to $w\in\hCW$ (the latter even as a module over the 
  cohomology). The map $i_x\colon\CB(w)_x\to\CB(w)^x$ then identifies
  with the adjunction map between the equivariant local
  hypercohomologies of the costalk and the stalk of the intersection
  cohomology complex on $\ol{BwB/B}$. The one-dimensional version
  $\ol\CB(w)_x\to\ol\CB(w)^x$ is then  obtained by reducing to the
  action of a one-dimensional subtorus of $T$ that corresponds to the
  coweight $\rho^\vee$. 

  Soergel then identifies the cokernel of $\ol\CB(w)_x\to\ol\CB(w)^x$ with
  the (ordinary) intersection cohomology of a certain
  projective variety $\ol Z$ in such a way
  that the multiplication with $t$ becomes a Lefschetz operator. The
  claim of the theorem hence is translated to complex algebraic
  geometry and equals the Hard Lefschetz theorem in this specific
  situation.
\end{proof}



\section{From characteristic zero to positive characteristic}\label{section-chzerochpos}

The main application of the Lefschetz theory developed in this article
is a calculation of an upper bound for the exceptional characteristics
for  Lusztig's conjecture. For this we use a base change argument. So  we have to develop a version of the Lefschetz theory
over the integers. It turns out that the Bott--Samelson modules can
indeed  be defined over $\DZ$ (cf.~Section \ref{sec-BSint}). Before we
introduce this construction, we associate to each Bott--Samelson
module its {\em Lefschetz data} and prove Proposition
\ref{prop-lefcoin} which serves us as a bridge from characteristic zero to positive characteristics.

\subsection{Lefschetz data}

Suppose that $f\colon A\to B$ is as in Section \ref{subsec-Lefcond}.

\begin{definition} Suppose that $(a_1,b_1)$, \dots, $(a_n,b_n)\in \DZ_{>0}\times \DZ$ are such that 
$$
\coker f \cong \bigoplus_{i=1}^n K[t]/\langle t^{a_i}\rangle\{b_i\}.
$$
We call the multiset $((a_1,b_1),\dots, (a_n,b_n))$ the {\em Lefschetz
  datum} for $f\colon A\to B$. 
\end{definition}

The following is immediate:
\begin{lemma}\label{lemma-lefconlefdat}
The map $f\colon A\to B$ satisfies the Lefschetz condition with center
$l$ if and only if its Lefschetz datum only involves pairs of the
form $(a,-l+a)$ for various $a>0$.
\end{lemma}
\subsection{The Lefschetz datum of a Bott--Samelson module}

It is from now on convenient to indicate the base
field $K$ in the notation of the relevant objects. So we 
denote by $S_K$ 
the symmetric algebra over the $K$-vector space $\hX\otimes_\DZ
K$, by $\CZ_K$ the structure algebra
over $K$, by $\CB_K(\bs)$ or $\CB_K(w)$  our special modules over
$\CZ_K$, by $\CB_K(\bs)^x$ or $\CB_K(w)^x$ the
corresponding stalks and so on. 

Let us fix an element $w\in\hCW$ and a sequence
$\bs=(s_1,\dots,s_j)$ such that $w=s_1\cdots s_l$ is a reduced expression. Suppose that the field $K$ is such
that $\ch\, K>N(w)=N(\bs)$. In particular, $(K,\bs)$ and $(K,w)$ are GKM-pairs. By Theorem \ref{theorem-BSspec}, the
Bott--Samelson module $\CB_K(\bs)$ decomposes into a direct sum of
modules of the form $\CB_K(x)$ up to degree shifts. We deduce from Lemma
\ref{lemma-inj} that for each $x\le w$ the map $\ol{\CB}_K(\bs)_x\to \ol{\CB}_K(\bs)^x$ is injective and that its cokernel is a torsion module.  We
call the Lefschetz datum of the map $\ol\CB_K(\bs)_x\to\ol\CB_K(\bs)^x$ the
{\em Lefschetz datum of $\CB_K(\bs)$ at $x$}.  The following result
shows how one can read off the decomposition pattern of $\CB_K(\bs)$
from its Lefschetz data.

\begin{lemma}\label{lemma-lefdata} Suppose that
$\CB_K(\bs)=\CB_K(w)\oplus\bigoplus_{i=1}^r\CB_K(x_i)\{l_i\}$. Choose $x<w$ and
let $((a_1,b_1),\dots,(a_n,b_n))$ be the Lefschetz datum of $\CB_K(\bs)$
at $x$. Then the number of $i$ such that $x_i=x$ equals the number of
$j$ such that $a_j=l(x)$.
\end{lemma}
\begin{proof}  By Proposition \ref{prop-lef} the Lefschetz datum of
  $\ol\CB_K(y)_x\to \ol\CB_K(y)^x$ involves only pairs $(a,b)$ with
  $a>l(x)$, unless $x=y$, in which case the cokernel of
  $\ol\CB_K(y)_x\to \ol\CB_K(y)^x$ is isomorphic to $K[t]/\langle
  t^{l(x)}\rangle$.
\end{proof}

\begin{proposition}\label{prop-lefcoin} Suppose that the Conjecture
  $\ref{conj-HL}$ holds for $\CB_K(w^\prime)$ for all $w^\prime<w$ and that the Lefschetz
  data of  $\CB_K(\bs)$ and $\CB_\DQ(\bs)$ coincide at all
  $x\le w$. Then Conjecture $\ref{conj-HL}$ holds for $\CB_K(w)$. 
\end{proposition}

\begin{proof} We know from Theorem \ref{theorem-charzero} that
  Conjecture \ref{conj-HL} holds for $\CB_\DQ(w)$, i.e.~that the map
  $\ol\CB_\DQ(w)_x\to\CB_\DQ(w)^x$ satisfies the Lefschetz condition
  with center $l(w)$. In view of Lemma \ref{lemma-lefconlefdat} it
  suffices for the proof of our claim  to show that the
  Lefschetz data of $\CB_K(w)$ and $\CB_\DQ(w)$ coincide at all
  $x<w$. 

  Lemma \ref{lemma-lefdata} together with our second assumption implies that the
  decomposition of $\CB_\DQ(\bs)$ into a direct sum of shifted copies
  of $\CB_\DQ(y)$ for various $y\le w$ parallels the analogous decomposition of
  $\CB_K(\bs)$. By our first assumption and Lemma \ref{lemma-lefconlefdat} the Lefschetz data of
  $\CB_\DQ(w^\prime)$ at
  all $x<w^\prime$ equals the corresponding Lefschetz data of
  $\CB_K(w^\prime)$ for all $w^\prime<w$. Since $\CB_K(w)$ occurs as a direct summand in
  $\CB_K(\bs)$ with multiplicity one and else only shifted $\CB_K(w^\prime)$
  with $w^\prime<w$ occur, the Lefschetz data of $\CB_K(w)$ and $\CB_\DQ(w)$
  coincide, which is what we wanted to show.
\end{proof}

\subsection{The main theorem}\label{subsec-maintheo}

We will now formulate the main result of this article, which gives an
upper bound
on the exceptional characteristics for the Conjecture \ref{conj-KL}. In order to state our result,  we need to
introduce the following numbers.

Let $\bs=(s_1,\dots,s_{l})$ be a sequence in $\hCS$.  Let us
define the corresponding Bott--Samelson element $\ul{H}(\bs)$ in the
affine Hecke algebra by 
$$
\ul{H}({\bs}):=\ul{H}_{s_1}\cdots\ul{H}_{s_l}.
$$
Then there are $x_1,\dots,x_n\in \hCW$ and $l_1,\dots,l_n\in\DZ$ with
$$
\ul{H}({\bs})=\sum_{i=1}^n v^{l_i}\ul{H}_{x_i}.
$$
 Let the
polynomials $a_x\in\DZ[v]$ be defined by 
$$
\ul{H}({\bs})=\sum_{x\in \hCW} a_x H_x.
$$
Set $r_x=a_x(1)$ and set
$$
r=r(\bs):=\max_{x}\{r_x\}.
$$
Let $d_x=\left(\frac{d}{dv}a_x\right)(1)$ be the sum of the exponents
of $a_x$, and set
$$
d=d(\bs):=\max_{x}\{d_x\}.
$$
Now we associate to $\bs$ the number
$U(\bs):=r!(r!(r-1)!N(\bs)^{l+2d})^r$ ($N(\bs)$ was defined in
Section \ref{sec-Leftheo}). We set 
$$
U(w):=\min_{\bs} U(\bs),
$$
where the minimum is taken over all reduced expressions for $w$. Note
that $U(w^\prime)\le U(w)$ for $w^\prime \le w$. 
Here is our result:
\begin{theorem} \label{theorem-bound} Let $w\in\hCW$ and suppose that $\ch
  K>U(w)$. Then the Conjectures $\ref{conj-KL}$
  and $\ref{conj-HL}$ hold for $\CB_K(w)$.
\end{theorem}

By Proposition \ref{prop-lefcoin}, Theorem \ref{theorem-bound} is a
consequence of the following result.

\begin{theorem}\label{theorem-lefdata} Let $\bs$ be a sequence in
  $\hCS$  and suppose that $\ch
  K> U(\bs)$. Then the Lefschetz data of
  $\CB_K(\bs)$ and $\CB_\DQ(\bs)$ coincide at all $x\in\hCW$. 
\end{theorem}

We prove the latter theorem by studying a base change. For this we have
to give an alternative construction of the Bott--Samelson modules that also makes
sense over $\DZ$.

\section{Another construction of the Bott--Samelson modules}\label{sec-BSint}

For a sequence  $\bs=(s_1,\dots,s_l)$ in $\hCS$ we let $I(\bs)$ be the
set of all subwords of the product $s_1\cdots s_l$, i.e.~
$$
I(\bs):=\{(i_1,\dots,i_n)\mid 0\le n\le l, 1\le i_1< \dots< i_n\le l\}.
$$
Then there is an evaluation map 
\begin{align*}
\ev\colon I(\bs) & \to \hCW,\\
(i_1,\dots,i_n)&\mapsto s_{i_1}\cdots s_{i_n}.
\end{align*}
If $l\ge 1$ then we denote by $\bs^\prime=(s_1,\dots, s_{l-1})$ the
truncated sequence. We consider $I(\bs^\prime)$ as a subset of
$I(\bs)$ in the obvious way. If we define for
$\sigma=(i_1,\dots,i_n)\in I(\bs^\prime)$ the element $\sigma s_l\in
I(\bs)$ by appending the reflection $s_l$ on the right, i.e.~$\sigma s_l:=(i_1,\dots,i_n,l)$, then $I(\bs)$ is the disjoint union of $I(\bs^\prime)$ and $I(\bs^\prime)s_l$.  

In the following we will denote by  $T$ a commutative unital domain in
which $2$ is not a zero divisor and which has furthermore the property
that the elements $\alpha\otimes
1\in\hX\otimes_\DZ T$ with $\alpha\in \hR$ are non-zero. The examples that we need are $T=\DZ$ and $T=K$, where $K$ is a field of characteristic $\ne 2$. We
denote by $S_T:=S_T(\hX\otimes_\DZ T)$ the symmetric algebra over the
free $T$-module associated to the affine weight lattice $\hX$. Again we consider $S_T$
as a graded algebra with $\deg \lambda=2$ for each non-zero
$\lambda\in \hX\otimes_\DZ T$. For $\lambda\in\hX$ we denote by
$\lambda$ also the image of $\lambda\otimes 1$ in $S_T$.
We  denote by $Q_T$ the $S_T$-algebra that is obtained from $S_T$ by
localization at the multiplicatively closed set generated by $2$ and the
set $\hR$, i.e.~
$$
Q_T:=S_T[2^{-1}][\alpha^{-1}, \alpha\in \hR].
$$ 

We define the structure algebra $\CZ_T$ in
a way analogous to the original definition in Section
\ref{subsec-strucalg}, i.e.~
$$
\CZ_T := \left\{ (z_x)\in \prod_{x\in \hCW}S_T\left| \, 
\begin{matrix} 
 z_x\equiv z_{s_{\alpha,n}x} \mod \alpha+n\delta \\
 \text{ for all $x\in \hCW$, $\alpha\in R$, $n\in\DZ$}
\end{matrix}
\right.\right\}.
$$
There is a natural action of $\CZ_T$  on $\bigoplus_{\sigma\in I(\bs)}
S_T$, given as follows: on the $\sigma$-component the
element $(z_x)\in\CZ_T$  acts as multiplication with $z_{\ev(\sigma)}$.

\subsection{Generically graded $S_T$-modules} 
We are now going to construct, for each $T$ as above and each sequence
$\bs=(s_1,\dots,s_l)$ in $\hCS$ a sub-$S_T$-module $\CX_T(\bs)$ of
$\bigoplus_{\sigma\in I(\bs)} S_T$. We call such a datum a {\em
  generically graded $S_T$-module}. 
We show  that $\CX_T(\bs)$ is stable under the action of
$\CZ_T$ on $\bigoplus_{\sigma\in I(\bs)} S_T$ that we defined above. In case $T=K$ is a field of characteristic $\ne
2$ we  show that  we have an isomorphism $\CX_K(\bs)\cong \CB_K(\bs)$ of $\CZ_K$-modules. 

The advantage of this new
construction is two-fold. Firstly, it provides a distinguished
$S_T$-basis of $\CX_T(\bs)$. Secondly, it allows us to view each element of
$\CX_T(\bs)$ as an $I(\bs)$-tuple of elements in $S_T$. We use this to
construct a basis of the
the stalk $\CB_K(\bs)^x$ for any $x$, which is then used to gain
arithmetic information on a  matrix describing the homomorphism $\CB_K(\bs)_x\to \CB_K(\bs)^x$.

Before we come to the definition of the generically graded modules we need some preparations.
For each sequence $\bs$ of length $l\ge 1$ let us denote by 
$$
\Delta\colon \bigoplus_{\gamma\in I(\bs^\prime)} S_T\to  \bigoplus_{\sigma\in I(\bs)} S_T
$$
the diagonal embedding, i.e.~ the map that sends  $(z_\gamma)$
to  $(z^\prime_\sigma)$ with $z^\prime_\sigma=z_\gamma$ if $\sigma\in
\{\gamma,\gamma s_l\}$.  We extend $\Delta$ to the diagonal embedding
$\bigoplus_{\gamma\in I(\bs^\prime)} Q_T\to  \bigoplus_{\sigma\in
  I(\bs)} Q_T$. 

Let $\iota\colon I(\bs)\to I(\bs)$ be the
fixed point free involution that interchanges $\gamma$ and $\gamma s_l$
for all $\gamma\in I(\bs^\prime)$. We denote by the same letter the
algebra isomomorphism $\iota\colon \bigoplus_{\sigma\in I(\bs)} S_T\to
\bigoplus_{\sigma\in I(\bs)} S_T$ that sends $(z_\sigma)$ to
$(z^\prime_\sigma)$ with $z^\prime_{\sigma}=z_{\iota(\sigma)}$. Then
the image of $\Delta$ coincides with the set of $\iota$-invariant
elements in $\bigoplus_{\sigma\in I(\bs)} S_T$. 

If $2$ is invertible in $T$ then we have a canonical decompositon
$$
\bigoplus_{\sigma\in I(\bs)}S_T=\left(\bigoplus_{\sigma\in
    I(\bs)}S_T\right)^\iota\oplus\left(\bigoplus_{\sigma\in
    I(\bs)}S_T\right)^{-\iota},
$$
where by $(\cdot)^\iota$ we denote the $\iota$-invariant elements and
by $(\cdot)^{-\iota}$ the $\iota$-anti-invariant elements. We have
already identified the first direct summand with $\bigoplus_{\gamma\in
  I(\bs^\prime)} S_T$ via $\Delta$. We also identify the second
summand with $\bigoplus_{\gamma\in
  I(\bs^\prime)} S_T$ via the map  $\Delta^{-}$ which send
$(z^\prime_\gamma)$ to $(z_\sigma)$ with
$z_\gamma=z^\prime_\gamma=-z_{\gamma s_l}$ (recall that we consider
$I(\bs^\prime)$ as a subset of $I(\bs)$). For an endomorphism  $f$  of $\bigoplus_{\gamma\in
  I(\bs^\prime)} S_T$ we denote by $\Delta^f$ the endomorphism on $\bigoplus_{\sigma\in
  I(\bs)} S_T$ that is diagonal with respect to the direct sum decomposition
above and is given by $f$ on both direct summands under the
identifications that we just constructed. Explicitely, for $(x,y)\in(\bigoplus_{\sigma\in
    I(\bs)}S_T)^\iota\oplus(\bigoplus_{\sigma\in
    I(\bs)}S_T)^{-\iota}$  we have $x=\Delta(x^\prime)$ and
  $y=\Delta^-(y^\prime)$ for some $x^\prime,y^\prime\in\bigoplus_{\gamma\in
    I(\bs^\prime)}S_T$ and we set 
$$
\Delta^f(x,y)=(\Delta(f(x^\prime)),\Delta^-(f(y^\prime))).
$$

For  $\lambda\in \hX$ define
$c^\lambda=(c^\lambda_x)\in\bigoplus_{x\in\hCW}S_T$  by
$c^\lambda_x=x(\lambda)\otimes 1\in \hX\otimes_\DZ T\subset S_T$. For
each $\alpha\in R$ and $n\in\DZ$ the linear form
$x(\lambda)-s_{\alpha,n}x(\lambda)$ vanishes on $\hH_{\alpha,n}$,
hence $x(\lambda)-s_{\alpha,n}x(\lambda)$ is proportional to
$\alpha+n\delta$, hence $c^{\lambda}\in \CZ_T$. We get a linear map
$\lambda\mapsto c^\lambda$ from $\hX$ to $\CZ_T$. Using the
$\CZ_T$-action  defined before we consider $c^\lambda$ as an
$S_T$-linear endomorphism on $\bigoplus_{\sigma\in I(\bs)} S_T$ for
all sequences $\bs$. 

Now we can define the generically graded $S_T$-modules.

\begin{definition} \label{def-defX} We define for all sequences $\bs$
  in $\hCS$ the sub-$S_T$-module  $\CX_T(\bs)\subset \bigoplus_{\sigma\in I(\bs)} S_T$ by the following algorithm:
\begin{enumerate}
\item  $\CX_T(\emptyset)\subset S_T$ is given by the identity $S_T=S_T$.
\item For $l\ge 1$ and $\bs=(s_1,\dots, s_l)$ let $\bs^\prime=(s_1,\dots, s_{l-1})$ be the truncated sequence. Then
\begin{align*}
\CX_T(\bs)&:=\Delta(\CX_T(\bs^\prime))+ c^{\alpha_l}\left(\Delta\left(\CX_T(\bs^\prime)\right)\right)\subset \bigoplus_{\sigma\in I(\bs)} S_T.
\end{align*}
Here $\alpha_l\in \widehat\Pi$ is the simple affine root associated to $s_l$.
\end{enumerate}
\end{definition}
We list some easy to deduce properties. 
\begin{lemma}\label{lemma-Xfree}  The sum in the inductive definition of $\CX_T(\bs)$ is direct, i.e.~
\begin{align*}
\CX_T(\bs)&=\Delta(\CX_T(\bs^\prime))\oplus c^{\alpha_l}\left(\Delta\left(\CX_T(\bs^\prime)\right)\right)\subset \bigoplus_{\sigma\in I(\bs)} S_T.
\end{align*}
In particular, $\CX_T(\bs)$ is a graded free $S_T$-module of rank
$(1+v^2)^l$.
\end{lemma}

\begin{proof} By construction the image of $\Delta$ consists of $\iota$-invariant elements. Since $c^{\alpha_l}_x=-c^{\alpha_l}_{xs_l}$, the image of $c^{\alpha_l}\circ\Delta$ consists of $\iota$-anti-invariant elements. The claim now follows from our general assumption that $2$ is not a zero divisor in $T$.
\end{proof}
The following lemma justifies the name ``generically graded module''.

\begin{lemma}\label{lemma-Xrootinv} The inclusion $\CX_T(\bs)\subset\bigoplus_{\sigma\in
    I(\bs)}S_T$ is an isomorphism after application of the functor
  $\cdot\otimes_{S_T}Q_T$. 
\end{lemma}
\begin{proof} This is immediate for $\bs=\emptyset$. It then follows
  inductively from the equation
  $\CX_T(\bs)=\Delta(\CX_T(\bs^\prime))\oplus
  c^{\alpha_l}\Delta(\CX_T(\bs^\prime))$ since after tensoring with $Q_T$ the first summand
  yields all invariant elements, while the second yields all
  anti-invariant elements in $\bigoplus_{\sigma\in I(\bs)} Q_T$, by the
  induction hypothesis (recall that  $2$ is invertible in $Q_T$).
\end{proof}
\subsection{Base change} Let $T\to T^\prime$ be a homomorphism of
unital rings which satisfy our general assumptions at the beginning of
Section \ref{sec-BSint}.

\begin{lemma}\label{lemma-bschX} The canonical isomorphism $\left(\bigoplus_{\sigma\in I(\bs)} S_T\right)\otimes_T T^\prime= \bigoplus_{\sigma\in I(\bs)} S_{T^\prime}$ identifies $\CX_T(\bs)\otimes_T T^\prime$ with $\CX_{T^\prime}(\bs)$.
\end{lemma}

\begin{proof}
 We  prove this by induction, the case $\bs=\emptyset$ being clear. If
 the length of $\bs$ is $\ge 1$, then we have, by definition and Lemma \ref{lemma-Xfree}, 
\begin{align*}
\CX_T(\bs)\otimes_T T^\prime & =\Delta(\CX_T(\bs^\prime))\otimes_T T^\prime \oplus c^{\alpha_l}\left(\Delta\left(\CX_T(\bs^\prime)\right)\right)\otimes_T T^\prime \\ 
&=\Delta(\CX_T(\bs^\prime)\otimes_T T^\prime)\oplus c^{\alpha_l}\left(\Delta\left(\CX_T(\bs^\prime)\otimes_T T^\prime\right)\right)\\ 
&=\Delta(\CX_{T^\prime}(\bs^\prime))\oplus c^{\alpha_l}\left(\Delta\left(\CX_{T^\prime}(\bs^\prime)\right)\right),
\end{align*}
which is the statement we wanted to prove. 
\end{proof}

\subsection{A duality} For an $S_T$-module $M$ we denote by  $\dual M=\Hom_{S_T}^\bullet(M, S_T)$ its graded dual.
For a sequence $\bs$ in $\hCS$ let us consider the standard pairing
\begin{align*}
\bigoplus_{\sigma\in I(\bs)} Q_T\times \bigoplus_{\sigma\in
  I(\bs)} Q_T & \to Q_T,\\
((x_\sigma),(y_\sigma)) &\mapsto \sum_\sigma x_\sigma y_\sigma.
\end{align*}

\begin{lemma}\label{lemma-Xdual} Using the above pairing we can canonically identify 
$$
\dual \CX_T(\bs) = \left\{(z_\sigma)\in\bigoplus_{\sigma\in I(\bs)} Q_T\left|
\,
\begin{matrix}
\sum_\sigma z_\sigma m_\sigma\in S_T \\
\text{ for all $(m_\sigma)\in\CX_T(\bs)$}
\end{matrix}
\right\}.
\right.
$$
\end{lemma}
\begin{proof}  Each $S_T$-linear homomorphism $\phi\colon \CX_T(\bs)\to S_T$
  induces, after application of the functor $\cdot\otimes_{S_T}Q_T$, a
  homomorphism $\phi\otimes 1\colon \bigoplus_{\sigma\in I(\bs)}
  Q_T\to Q_T$, by Lemma \ref{lemma-Xrootinv}. After we identified $\Hom_{Q_T}(\bigoplus_{\sigma\in I(\bs)}
  Q_T, Q_T)$ with $\bigoplus_{\sigma\in I(\bs)}
  Q_T$ in the obvious way, we have in fact associated to $\phi$ an
  element in the set on the right hand side of the lemma's claim. A moment's thought shows that this yields a bijection. \end{proof}

Now we come to a definition that is very similar to the definition of
the $\CX_T(\bs)$'s.
\begin{definition} \label{def-defY} We define for all sequences $\bs$
  in $\hCS$ the sub-$S_T$-module  $\CY_T(\bs)\subset \bigoplus_{\sigma\in I(\bs)} Q_T$ by the following algorithm:
\begin{enumerate}
\item  $\CY_T(\emptyset)\subset Q_T$ is given by the inclusion
  $S_T\subset Q_T$.
\item For $l\ge 1$ and $\bs=(s_1,\dots, s_l)$ let $\bs^\prime=(s_1,\dots, s_{l-1})$ be the truncated sequence. Then
\begin{align*}
\CY_T(\bs)&:=\Delta(\CY_T(\bs^\prime))+ (c^{\alpha_l})^{-1}\left(\Delta\left(\CY_T(\bs^\prime)\right)\right)\subset \bigoplus_{\sigma\in I(\bs)} Q_T.
\end{align*}
Here $\alpha_l\in \widehat\Pi$ is the simple affine root associated to $s_l$.
\end{enumerate}
\end{definition}
Note that $\ev(\sigma)(\alpha_l)\in\hX$ is an affine root, so  its
image in $Q_T$ is invertible. Hence $c^{\alpha_l}$ is an automorphism of
$\bigoplus_{\sigma\in I(\bs)}Q_T$, so is invertible. As in Lemma \ref{lemma-Xfree} we show that the sum in part (2) of the definition is direct, i.e.~ we have
\begin{lemma}\label{lemma-Yfree} For each sequence $\bs$ of length $\ge 1$ we have $\CY_T(\bs)=\Delta(\CY_T(\bs^\prime))\oplus (c^{\alpha_l})^{-1}\left(\Delta\left(\CY_T(\bs^\prime)\right)\right)\subset \bigoplus_{\sigma\in I(\bs)} Q_T$.
\end{lemma}
Now we identify $\CY_T(\bs)$ with the dual of $\CX_T(\bs)$. 
\begin{lemma}\label{lemma-XYdual} Suppose that $2$ is invertible in $T$. For each sequence $\bs$ we then have
$$
\CY_T(\bs)=\left\{(z_\sigma)\in\bigoplus_{\sigma\in I(\bs)} Q_T\left|
\,
\begin{matrix}
\sum_\sigma z_\sigma m_\sigma\in S_T \\
\text{ for all $(m_\sigma)\in\CX_T(\bs)$}
\end{matrix}
\right\}.
\right.
$$
Hence, using Lemma \ref{lemma-Xdual} we get an identification $\dual\CX_T(\bs)=\CY_T(\bs)$.
\end{lemma}
 \begin{proof}  For $\bs=\emptyset$ the claim follows immediately from
  the definitions. So suppose that the length of $\bs$ is at least one
  and that the statement holds for the truncated sequence
  $\bs^\prime$. We have $\CX_T(\bs)=\Delta(\CX_T(\bs^\prime))\oplus c^{\alpha_l}\Delta(\CX_T(\bs^\prime))$ and $\CY_T(\bs)=\Delta(\CY_T(\bs^\prime))\oplus (c^{\alpha_l})^{-1}\Delta(\CY_T(\bs^\prime))$.  For
  $x\in\CX_T(\bs^\prime)$ and $y\in\CY_T(\bs^\prime)$ we have
\begin{align*}
(\Delta(x),\Delta(y))&=2(x,y),\\
(c^{\alpha_l}\Delta(x), (c^{\alpha_l})^{-1}\Delta(y))&=2(x,y).
\end{align*}
Since $2$ is supposed to be invertible in $T$ we deduce from the
induction hypothesis that $(\cdot,\cdot)$ pairs either
$\Delta(\CX_T(\bs^\prime))$ and
$\Delta(\CY_T(\bs^\prime))$ or  $c^{\alpha_l}\Delta(\CX_T(\bs^\prime))$ and
$(c^{\alpha_l})^{-1}\Delta(\CX_T(\bs^\prime))$ perfectly. 

On the other hand,  $c^{\alpha_l}$ acts by a diagonal multiplication
with a $\iota$-anti-invariant element, and we see that 
\begin{align*}
(\Delta(\CX_T(\bs^\prime)),
(c^{\alpha_l})^{-1}\Delta(\CY_T(\bs^\prime)) &= 0,\\
(c^\alpha\Delta(\CX_T(\bs^\prime)),
\Delta(\CY_T(\bs^\prime))&=0.
\end{align*}
\end{proof}

\subsection{Stalks and costalks}

The inclusion $\CX_T(\bs)\subset\bigoplus_{\sigma\in I(\bs)} S_T$ 
allows us to define stalks and costalks of $\CX_T(\bs)$  in a way
analogous to the case of $\CZ$-modules (cf.~ Section
\ref{subsec-gendecomp}). Let $\CI$ be
a subset of $\hCW$ and define $I(\bs)_\CI\subset I(\bs)$ as
the set of elements whose evaluation is contained in $\CI$, i.e. $I(\bs)_\CI=\ev^{-1}(\CI)$. We view
$\bigoplus_{\gamma\in I(\bs)_\CI} S_T$ as a direct summand in
$\bigoplus_{\gamma\in I(\bs)} S_T$. For a subspace 
$M$ of $\bigoplus_{\gamma\in I(\bs)}S_T$ we define 
\begin{align*}
M_\CI & := M\cap \bigoplus_{\gamma\in I(\bs)_\CI} S_T\\ 
M^\CI & := \im\left(M\subset \bigoplus_{\gamma\in I(\bs)} S_T\to \bigoplus_{\gamma\in I(\bs)_\CI} S_T\right).
\end{align*}
In particular, we define the {\em stalk} of $M$ at $x$ as
$M^x:=M^{\{x\}}$ and the {\em costalk} of $M$ at $x$ as $M_x:=M_{\{x\}}$.

\begin{lemma}\label{lemma-dualstalkX} Let $\bs$ be a sequence in $\hCS$ and $x\in\hCW$. The
  identification in Lemma \ref{lemma-Xdual} induces an identification
\begin{align*}
\dual(\CX_T(\bs)^x) &=  \left\{(z_\sigma)\in\bigoplus_{\sigma\in I(\bs)_x} Q_T\left|
\,
\begin{matrix}
\sum_\sigma z_\sigma m_\sigma\in S_T \\
\text{ for all $(m_\sigma)\in\CX_T(\bs)^x$}
\end{matrix}
\right\}
\right.\\
&= \CY_T(\bs)_x.
\end{align*}
\end{lemma}
\begin{proof} Note that $\dual (\CX_T(\bs)^x)\subset\dual \CX_T(\bs)$
  is the subspace of all linear forms $\phi\colon \CX_T(\bs)\to S_T$ that
  vanish on the kernel of the map $\CX_T(\bs)\to\CX_T(\bs)^x$. Using
  Lemma \ref{lemma-Xrootinv} we see that this is identified with the
  set of $(z_\sigma)\in\bigoplus_{\sigma\in I(\bs)}Q_T$ with
  $z_\sigma=0$ if $\sigma\not\in I(\bs)_x$. From this we get the first
  identity. The second identity follows from Lemma \ref{lemma-XYdual}.
\end{proof}

\subsection{Distinguished bases}
By Lemma \ref{lemma-Xfree} and Lemma \ref{lemma-Yfree}, the $S_T$-modules $\CX_{T}(\bs)$ and $\CY_T(\bs)$ are free of 
(ungraded) rank $2^l$.

\begin{definition} Let $E_T(\bs)\subset \CX_T(\bs)$ and $F_T(\bs)\subset \CY_T(\bs)$ be the homogeneous $S_T$-bases that are defined inductively by the following algorithms:
\begin{enumerate}
\item $E_T(\emptyset)=\{1\}\subset \CX_T(\emptyset)=S_T$ and $F_T(\emptyset)=\{1\}\in \CY_T(\emptyset)=S_T$.
\item If $\bs$ is a sequence of length $l\ge 1$, then 
\begin{align*}
E_T(\bs)&:= \Delta(E_T({\bs^\prime}))\cup c^{\alpha_l}\Delta(E_T(\bs^\prime)),\\
F_T(\bs)&:=\Delta(F_T({\bs^\prime}))\cup (c^{\alpha_l})^{-1}\Delta(F_T({\bs^\prime})).
\end{align*}
\end{enumerate}
\end{definition}
Note that for a homomorphism  $T\to T^\prime$  of unital domains  the
identification $\CX_T(\bs)\otimes_T T^\prime=\CX_{T^\prime}(\bs)$ of
Lemma \ref{lemma-bschX} identifies $E_T(\bs)\otimes 1$ with
$E_{T^\prime}(\bs)$.

Using the inclusion $\CX_T(\bs)\subset \bigoplus_{\sigma\in I(\bs)} S_T$ we can view each element of $E_T({\bs})$ as an $I(\bs)$-tuple  of elements  in $S_T$. The following statement is clear from the definition.

\begin{lemma} Each coordinate of each element $e$ of $E_T(\bs)$ is a product of $1/2\deg e$ roots. 
\end{lemma}

We need the following endomorphism $P(\bs)$ on $\bigoplus_{\sigma\in I(\bs)}Q_T$.

\begin{definition} Let $P(\bs)$ be defined for any sequence $\bs$ in
  $\hCS$ inductively by the
  following algorithm:
\begin{enumerate}
\item We let $P(\emptyset)$ be the identity on $Q_T$.
\item If $P(\bs^\prime)$ is already defined then we set 
$$
P({\bs}):=c^{\alpha_l}\circ \Delta^{P({\bs^\prime})}.
$$
\end{enumerate}
\end{definition}

By definition, $P(\bs)$ acts diagonally and on each direct summand
it is given by multiplication with a product of $l$ roots, where $l$
is the length of $\bs$.

\begin{lemma} \label{lemma-EandF}
We have $E_T({\bs})= P(\bs)(F_T(\bs))$. In particular, we have $\CX_T(\bs)=P(\bs)(\CY_T(\bs))$.
\end{lemma}
\begin{proof} Again we use induction on the length of $\bs$. For $\bs=\emptyset$ we have $E_T(\emptyset)=F_T({\emptyset})=\{1\}$ and $P(\emptyset)=\id$. Suppose that $\bs$ is of length $l\ge 1$ and that the statement is true for the truncated sequence $\bs^\prime$. Then 
\begin{align*}
P({\bs})(\Delta(F_T({\bs^\prime})) & =c^{\alpha_l}\Delta^{P({\bs^\prime})}(\Delta(F_T({\bs^\prime})) \\
&= c^{\alpha_l}(\Delta(P(\bs^\prime)(F_T(\bs^\prime)))) \\
&= c^{\alpha_l}(\Delta(E_T(\bs))).
\end{align*}
Analogously,
\begin{align*}
P({\bs})((c^{\alpha_l})^{-1}\Delta(F_T({\bs^\prime}))) &= c^{\alpha_l}\Delta^{P({\bs^\prime})}((c^{\alpha_l})^{-1}\Delta(F_T({\bs^\prime})))  \\
&= \Delta(P(\bs^\prime)(F_T(\bs^\prime)))\\
&= \Delta(E_T(\bs))
\end{align*}
and the claim follows.
\end{proof}

\subsection{The relation to Bott-Samelson modules}
Now we show that we indeed have given an alternative construction of
the Bott--Samelson modules.
\begin{proposition}\label{prop-BSX}
\begin{enumerate}
\item
Suppose that $2$ is invertible in $T$. Then for all sequences $\bs$, the subset $\CX_T(\bs)$ of
  $\bigoplus_{\sigma\in I(\bs)} S_T$ is stable under the action of
  $\CZ_T$ defined above.
\item If $T=K$ is a field of characteristic $\ne 2$, then $\CX_K(\bs)$
  identifies with the Bott-Samelson module $\CB_K(\bs)$ as a
 $\CZ_K$-module.
\end{enumerate}
\end{proposition}

\begin{proof}  We prove the proposition by induction on the length of
  $\bs$. For $\bs=\emptyset$ both statements immediately follow from
  the definitions. So suppose that $\bs=(s_1,\dots, s_l)$ with $l\ge 1$ and that
  we have proven the claims for $\bs^\prime$. Now note that $\Delta\colon
\bigoplus_{\gamma\in I(\bs^\prime)}S_T\to \bigoplus_{\sigma\in I(\bs)} S_T$
is actually a homomorphism of $\CZ_T^{s_l}$-modules. From our induction
assumption we hence get that $\Delta(\CX_T(\bs^\prime))$ and
$c^{\alpha_l}(\Delta(\CX_T(\bs^\prime)))$ are stable under the action
of $\CZ_T^{s_l}$. 

By Lemma 5.1 in \cite{FieTAMS} we have
$\CZ_T=\CZ_T^{s_l}\oplus c^{\alpha_l}\CZ_T^{s_l}$ under the assumption
that $2$ is invertible in $T$. Note that
$(c^{\alpha_l})^2\in\CZ_T^{s_l}$. Hence giving a $\CZ_T$-module
structure on an abelian group $M$ is the same as giving a
$\CZ_T^{s_l}$-module structure on $M$ together with a
$\CZ_T^{s_l}$-linear endomorphism $f$ on $M$ such that $f^2$ coincides
with the action of $(c^{\alpha_l})^2$. From this, statement (1)
for $\CX_T(\bs)=\Delta(\CX_T(\bs^\prime)\oplus c^{\alpha_l}\Delta(\CX_T(\bs^\prime))$ easily follows by defining the map $f$ on the
components as $\Delta(\CX_T(\bs^\prime))\to c^{\alpha_l}\Delta(\CX_T(\bs^\prime))$,
$x\mapsto c^{\alpha_l}x$ and $c^{\alpha_l}\Delta(\CX_T(\bs^\prime))\to \Delta(\CX_T(\bs^\prime))$,
$c^{\alpha_l}x\mapsto (c^{\alpha_l})^2x$. 
Similarly,  statement (2) follows after we
considered the  identities
\begin{align*}
\CB_K(\bs)=\CZ\otimes_{\CZ^{s_l}}\CB_K(\bs^\prime) & =  (\CZ_K^{s_l}\oplus
c^{\alpha_l}\CZ_K^{s_l})\otimes_{\CZ^{s_l}}\CB_K(\bs^\prime) \\
& =  \CB_K(\bs^\prime)\oplus c^{\alpha_l}\CB_K(\bs^\prime).
\end{align*}
\end{proof}

Finally, we compare the stalks of $\CB_K(\bs)$ and the stalks of $\CX_K(\bs)$.

\begin{lemma}\label{lemma-stalksX} Let $\bs$ be a sequence in $\hCS$ and $x\in\hCW$. 
 If $T=K$ is a field of characteristic $\ne 2$, then the
  isomorphism $\CX_K(\bs)\cong\CB_K(\bs)$ of $\CZ_K$-modules induces
  isomorphisms
$\CX_K(\bs)_x\cong\CB_K(\bs)_x$  and
$\CX_K(\bs)^x\cong\CB_K(\bs)^x$.
\end{lemma}
\begin{proof} The statement follows immediately from the definition of
  the action of $\CZ_K$ on
  $\CX_K(\bs)$ and Lemma \ref{lemma-stalks}.
 \end{proof}

\subsection{Integral matrices}
Now we want to prove Theorem \ref{theorem-lefdata}, so we fix a
sequence $\bs=(s_1,\dots,s_l)$ in $\hCS$ and a field $k$ of positive
characteristic such that $\ch\,k >U(\bs)$. Note that this implies
that  $(k,\bs)$ is a GKM-pair.   Let us
also fix an element $x$ of $\hCW$ that is contained in
$\ev(I(\bs))$. We want to prove that the Lefschetz data of
$\CB_k(\bs)$ and $\CB_\DQ(\bs)$ coincide at $x$. 
For this we construct a matrix $X$ with entries in $S_\DZ$ that
describes both the inclusions $\CB_k(\bs)_x\to\CB_k(\bs)^x$ and
$\CB_\DQ(\bs)_x\to\CB_\DQ(\bs)^x$ and then give an estimate on the
minors of $X$. 

In the following $K$
denotes either the field $k$ or the rational numbers $\DQ$.
Recall that we have a distinguished basis $E_\DZ(\bs)$ of
$\CX_\DZ(\bs)$. We next want to choose a subset $E(x)$ of $E_\DZ(\bs)$
that has the property that the images of its elements in both
$\CX_k(\bs)^x$ and $\CX_\DQ(\bs)^x$ form a basis. (We consider  the
natural maps $\CX_\DZ(\bs)\to\CX_\DZ(\bs)\otimes_\DZ K=\CX_K(\bs)\to\CX_K(\bs)^x$ for $K=k$ and
$K=\DQ$).

\begin{lemma}\label{lemma-extbas} There is a subset $E(x)$  of $E_\DZ(\bs)$  with
  the property that the images of
 its elements  in $\CX_{K}(\bs)^x$ form an $S_K$-basis for both $K=k$ and $K=\DQ$.
\end{lemma}
\begin{proof} The images of the elements of $E_\DZ(\bs)$ generate the
  $S_k$-module $\CX_k(\bs)^x$. Since this is a graded free
  $S_k$-module we can find a subset $E(x)$ of $E(\bs)$ such
  that its image forms a basis. Then $E(x)$ is $S_\DZ$-linearly independent
  in $\CX_\DZ(\bs)^x$ since we have a commutative diagram

\centerline{
\xymatrix{
\CX_\DZ(\bs)\ar[d]\ar[r] &\CX_\DZ(\bs)^x\ar[d]\\
\CX_k(\bs)\ar[r]&\CX_k(\bs)^x.
}
}
\noindent 
Hence $E(x)$ is $S_\DQ$-linearly independent in
  $\CX_\DQ(\bs)^x$. Since the graded ranks of $\CX_k(\bs)^x$ and
  $\CX_\DQ(\bs)^x$ coincide, by Lemma \ref{lemma-stalksX} and Lemma
  \ref{lemma-rankind}, the image of  $E(x)$ forms a basis of $\CX_\DQ(\bs)^x$
  as well.
\end{proof}

We now fix such a subset $E(x)$ of $E_\DZ(\bs)$. From Lemma
\ref{lemma-Xrootinv} we deduce that the ungraded rank of
$\CX_K(\bs)^x$ equals the order of $I(\bs)_x$ which coincides with the
number $r_x$ defined in Section \ref{subsec-maintheo}. Let us identify
$I(\bs)_x$ with $\{1,\dots,r_x\}$ and let us choose an enumeration
$\{v_1,\dots,v_{r_x}\}$ of the images of the
elements of $E(x)$ in $\CX_\DZ(\bs)^x$. Let $E\subset
S_\DZ^{r_x\times r_x}$ be the matrix $(v_1,\dots,v_{r_x})^T$, i.e.~the
matrix with row vectors $v_1,\dots, v_{r_x}$.

We apply the base change $S_\DZ\to S_K$ and consider the matrix
$E$ now as an element in $S_K^{r_x\times r_x}$. Its row vectors  form a basis
of $\CX_K(\bs)^x$ by construction. By Lemma \ref{lemma-Xrootinv}, the matrix
$E$ is invertible over $Q_K$. The row vectors of the matrix
$(E^{-1})^T\in Q_K^{r_x\times r_x}$ now form a basis of the space
$$
\left\{(z_\sigma)\in\bigoplus_{\sigma\in I(\bs)_x} Q_K\left|
\,
\begin{matrix}
\sum_\sigma z_\sigma m_\sigma\in S_K \\
\text{ for all $(m_\sigma)\in\CX_K(\bs)$}
\end{matrix}
\right\}
\right.
$$
which we have identified, as a subspace of $\bigoplus_{\sigma\in I(\bs)_x} Q_K$, with $\CY_K(\bs)_x$ in Lemma \ref{lemma-dualstalkX}. 

Recall the endormorphism
$P(\bs)$ on $\bigoplus_{\sigma\in I(\bs)}
Q_\DZ$. It acts diagonally and so
we can consider its restriction $P(\bs)^x$ to the direct summand
$\bigoplus_{\sigma\in I(\bs)_x} Q_\DZ$. We consider the latter
endomorphism as   a diagonal $r_x\times r_x$-matrix which we denote by $P$
for simplicity. 
Now Lemma \ref{lemma-dualstalkX} and Lemma \ref{lemma-EandF} show that the row vectors
$w_1,\dots,w_{r_x}$ of the matrix $(E^{-1})^TP$ form a basis of the
$S_K$-module $\CX_K(\bs)_x$. 

Hence we have constructed bases $(w_i)$ and $(v_i)$ of
$\CX_K(\bs)_x$ and $\CX_K(\bs)^x$. Now consider the matrix
$$
X=(x_{ij}):=(E^{-1})^T P E^{-1}.
$$
Then we have
$$
w_i=\sum_{j=1}^{r_x} x_{ij} v_j.
$$
The matrix $X$ already appears in H\"arterich's paper \cite{Haert} in
which he studies the push-forward of the equivariant constant sheaf on a
Bott-Samelson variety to a Schubert variety, cf.~ the remark following
Proposition 6.8 in {\em loc.~cit.}

Recall that we defined $\tau\colon S_\DZ\to \DZ[t]$  by the property
$\tau(\alpha)=t$ for all simple affine roots $\alpha$. We denote by $\tau$
 also the homomorphism $S_K\to K[t]$ obtained by base
change.  Recall that we assume $\ch\, k> N(\bs)$. In order to determine the Lefschetz data we have to
study the inclusion $\ol{\CX}_K(\bs)_x\subset\ol\CX_K(\bs)^x$ which is
given by the image $\ol X$ of the matrix $X$ in $K[t]^{r_x\times r_x}$
under the map $\tau$. In order to prove Theorem \ref{theorem-lefdata} it is  enough to prove the following claim:

{\em Under the assumption $\ch\, k>U(\bs)$, a minor of
the matrix $X$ vanishes in $k[t]$ if and only if it vanishes in
$\DQ[t]$.
}

For convenience we study, instead of $X$,  the matrix
$$
X^\prime=(x_{ij}^\prime):=\det(E)^2X^\prime= (\det E) ( E^{-1})^T P (\det  E)
 E^{-1}.
$$
It has the advantage that its entries are
elements in $S_\DZ$, and it is certainly enough to prove the above
claim for $X^\prime$ instead of $X$.

Let us denote by $d_i$ the half degree of $v_i$. Set
$d=\sum_{i=1}^{r_x} d_i$. Then we have:
\begin{lemma}\label{lemma-estxij} The element $x^\prime_{ij}\in S_\DZ$ can be written as a sum of
  $r_x!(r_x-1)!$ products of $2d-d_i-d_j+l$ roots.
\end{lemma}
\begin{proof} Recall that the $ij$-entry $v_{ij}$ of the matrix
  $E$ is a product of $d_i$ roots. The $ij$-entry of $(\det
  E) E^{-1}$ is the determinant of the matrix obtained from $E$ by deleting the $j$-th row and the $i$-th column. Hence it can be written as a sum of $(r_x-1)!$ products of
  $d-d_j$ roots. It follows that the $ij$-entry of $P(\det
  E) E^{-1}$ can  be written as a sum of $(r_x-1)!$ products of
  $d-d_j+l$ roots. Finally, the $ij$-entry $x^\prime_{ij}$ of $X^\prime$ can be
  written as a sum of $r_x(r_x-1)!^2=r_x!(r_x-1)!$ products of $2d-d_i-d_j+l$ roots. 
\end{proof}

Let $J$ and $J^\prime$ be subsets of
$\{1,\dots, r_x\}$ of the same size $s$ and denote by
$\xi_{JJ^\prime}\in S_\DZ$ the corresponding minor of the matrix
$X^\prime$. Set $d_J=\sum_{i\in J} d_i$ and $d_{J^\prime}=\sum_{i\in
  J^\prime}d_i$.  Then we have:

\begin{lemma}\label{lemma-minors} The element $\xi_{JJ^\prime}$ can be written as a sum of
  $s!(r_x!(r_x-1)!)^s$ products of  $s(2d+l)-d_J-d_{J^\prime}$ roots. 
\end{lemma}
\begin{proof} Each entry of the submatrix of $X^\prime$ corresponding to
  $J$ and $J^\prime$ is a sum of $r_x!(r_x-1)!$ products of roots by Lemma
  \ref{lemma-estxij}. Hence $\xi_{JJ^\prime}$ can be written as a sum
  of $s!(r_x!(r_x-1)!)^s$ products of roots. Moreover, the $ij$-entry of
  $X^\prime$ is a sum of products of $2d+l-d_i-d_j$ roots. So each
  summand in $\xi_{JJ^\prime}$ is a product of
  $s(2d+l)-d_J-d_{J^\prime}$ roots.
\end{proof}

One should keep in mind that the above estimate is very
rough. For example, up to a power of $2$, the determinant of the
matrix $X^\prime$ is a product of roots by Lemma \ref{lemma-Xrootinv}!

The map $\tau\colon S_\DZ\to\DZ[t]$ maps a positive root $\alpha$ to
$\height(\alpha)t$. By definition of $N(\bs)$ we have $\height(\alpha)\le
N(\bs)$ for all roots that occur in the situation of Lemma
\ref{lemma-minors}. If we denote by $\ol \xi_{JJ^\prime}$ the image
of $\xi_{JJ^\prime}$ in $\DZ[t]$, then we can deduce from Lemma
\ref{lemma-minors} that
$\ol\xi_{JJ^\prime}=a_{JJ^\prime}t^{s(2d+l)-d_J-d_{J^\prime}}$ with
$$
|a_{JJ^\prime}|\le s!(r_x!(r_x-1)!)^s
N(\bs)^{s(2d+l)-d_J-d_{J^\prime}}.
$$
Since $r_x\le r$, $s\le r$ and $d_J, d_{J^\prime}\ge 0$ we have
$$
|a_{JJ^\prime}|\le r!(r!(r-1)!N(\bs)^{2d+l})^r=U(\bs).
$$

Hence, for $\ch\,k>U(\bs)$ the minor $\xi_{JJ^\prime}$ vanishes in
$k[t]$ if and only if it vanishes in $\DQ[t]$, which is the statement
we wanted to show in order to prove Theorem \ref{theorem-lefdata}.


\end{document}